\documentclass[a4paper,11pt]{article}
\usepackage{epsfig}
\usepackage{amsmath,amsthm,amssymb,amscd,graphicx,psfrag,epic,epsfig}
\usepackage{fancyhdr}
\usepackage{amsfonts}
\usepackage[frenchb,english]{babel}

\newtheorem{thm}{Theorem}[section]

\newtheorem{lem}[thm]{Lemma}
\newtheorem{prop}[thm]{Proposition}

\newtheorem{df}[thm]{Definition}
\newtheorem{rem}[thm]{Remark}

\newtheorem{conv}{Convention}
\numberwithin{equation}{section}

\newcommand\<{\langle}
\renewcommand\>{\rangle}

\newcommand{\A}{\mathcal {A}}\newcommand{\B}{\mathcal {B}}\newcommand{\C}{\mathcal {C}}
\renewcommand{\d}{\tilde D}

\title{Global existence for the Euler-Maxwell system}
\author{P. Germain, N. Masmoudi}

\begin{document}

\maketitle

\begin{abstract}
The Euler-Maxwell system  
describes the evolution of a plasma when the collisions 
are important enough that each species is in a hydrodynamic 
equilibrium. In this paper we prove global existence of 
small solutions to this system set in the whole three-dimensional space, by combining the space-time resonance method, 
dispersive estimates, localization estimates  and energy estimates. An important novelty is that we can 
 prove a very slow growth ($t^{C\epsilon}$)  of high derivatives even with only $\frac1{\sqrt t}$ decay of the $L^\infty$ norm
by reiterating the energy estimate. 
\end{abstract}
 
\begin{quote}
\footnotesize\tableofcontents
\end{quote}

\section{Introduction}

\subsection{Plasma physics and Euler-Maxwell}

There are different models to describe the state of a  plasma depending on 
several  parameters such as the Debey length, the plasma frequency, the collision 
frequencies between the different species... 
Formal derivation of these models can be found in Plasma Physics textbooks 
(see for instance 
 Bellan \cite{Bellan06},  Boyd and Sanderson \cite{BS03}, Dendy \cite{Dendy90} and the 
paper \cite{BCDDGT04} ...)
 
Since the plasma consists of a very large number of interacting particles, 
it is appropriate to adopt a statistical approach to describe it.
In the kinetic description, 
it is only necessary to evolve the distribution function $f_\alpha (t,x,v)$ for 
each species in the system. Vlasov equation is used in this case with the 
Lorentz force term and a collision term. It is 
 coupled with the  Maxwell equations 
for the electromagnetic  fields.

If collisions are important, then each species is in a local equilibrium 
and the plasma is  treated as a fluid. More precisely it  is treated as a mixture 
of two or more interacting fluids. This is the two-fluid model or the 
so-called  Euler-Maxwell system.  We refer  to \cite{LM01arma1,GS04,LM10}
for more about  hydrodynamic limits. 
Another level of approximation  consists in treating the plasma as a single fluid 
by using the fact that the mass of the electrons is much smaller than
the mass of the ions.  This is the model which we are going to consider in this paper. 

\subsection{The Euler-Maxwell equation}

The Cauchy problem for the one fluid version of the Euler-Maxwell system reads
\begin{equation} \label{em}
\left\{ 
\begin{array}{l}
\rho \left( \partial_t u + u\cdot \nabla u \right) = -\frac{p'(\rho)}{m} \nabla \rho - \frac{e \rho}{m}\left( E + \frac{1}{c} u \times B \right) \\
\partial_t \rho + \nabla \cdot (\rho u) = 0 \\
\partial_t B + c \nabla \times E = 0 \\
\partial_t E - c \nabla \times B = 4 \pi e \rho u \\
\nabla \cdot E = 4 \pi e (\bar \rho - \rho) \\
\nabla \cdot B = 0 \\
(u,\rho,E,B)(t=0) = (u_0,\rho_0,E_0,B_0).
\end{array}
\right.
\end{equation}
The unknown functions are: $\rho$, the density of electrons; $u$, the average velocity of the electrons; $E$, the electric field; $B$ the magnetic field. The physical constants are: $c$, the speed of light; $e$, the charge of the electron; $m$, the mass of the electron. Finally, $\bar \rho$ is the uniform density of ions, and the electron gas is supposed to be barotropic, the pressure being given by $p(\rho)$.  

Let us first recall a few results related to \eqref{em}. Global existence of weak solutions was obtained for 
a related 1d model in \cite{CJW00} using compensated compactness. Also, several asymptotic problems (W{KB} asymptotics, incompressible limit, non-relativistic limit, quasi-neutral limit...) were studied
to derive simplified models starting from  the Euler-Maxwell system  \cite{Texier05,YWLL10,YW09,PW08}. We also refer to 
\cite{Mamousi10jmpa} where the incompressible Navier-Stokes system is studied. 

Going back to our system  \eqref{em}, 
we notice that the two last equations above can be removed, as soon as they are satisfied at the initial time, {\em which we assume from now on}: they are then conserved by the flow given by the first four.

\subsection{Vicinity of the trivial equilibrium state}

An obvious equilibrium state of the above system is $(\rho,u,E,B) = (\bar \rho,0,0,0)$. In order to study its stability, it is instructive to linearize the above system, and compute evolution equations for its unknowns. It is convenient then to split $u$ into its divergence-free part $Pu$, and its curl-free part $Qu$: $u= Pu + Qu$. Similarly, $E = PE + QE$. The obtained system reads

\begin{equation*}
\left\{ 
\begin{array}{l}
\left( \partial_t^2 - c_s^2 \Delta + \omega_p^2 \right) \left( \begin{array}{c} QE \\ \rho - \bar \rho \\ Qu \end{array} \right) = 0 \\
\left( \partial_t^2 - c^2 \Delta + \omega_p^2 \right) \left( \begin{array}{c} PE \\ \nabla \times B + \frac{4\pi e \bar \rho}{c} Pu \end{array} \right) = 0 \\
\partial_t \left( B - \frac{cm}{e} \nabla \times u \right) = 0 \\
\end{array}
\right.
\end{equation*}
where the speed of sound $c_s$ and the plasma frequency $\omega_p$ are given by
$$
c_s = \sqrt{ \frac{p'(\bar \rho)}{m}} \;\;\;\;\mbox{and} \;\;\;\;\omega_p = \sqrt{\frac{4 \pi e^2 \bar \rho}{m}}.
$$
Thus around the equilibrium, and at a linear level, some unknowns are governed by Klein Gordon equation (with different speeds), whereas the quantity $B - \frac{cm}{e} \nabla \times u$ is conserved. The Klein Gordon equations entail decay, which is one of the keys of the global stability result which we will prove; as for the quantity $B - \frac{cm}{e} \nabla \times u$, no decay is to be expected a priori. We will therefore set it to zero, which, as it turns out, is conserved by the nonlinear flow.

\subsection{Adimensionalization and reductions}

In the following, we set for simplicity the physical constants $m$, $e$, $c$, as well as $\bar \rho$ to 1. We also drop the $4\pi$ factors, since they are irrelevant. However $
c_s^2 = p'(\bar \rho)=p'(1)$ remains a number less than 1. In order to simplify a little bit the estimates, we assume 
$$
p(\rho) \overset{def}= \frac{c_s^2}{3} \rho^3.
$$
Finally, set
$$
n \overset{def}{=} \rho-1.
$$
The Cauchy problem becomes
\begin{equation*}
(EM)\;\;\;\;\;\;\;\;
\left\{ 
\begin{array}{l}
\partial_t u + u\cdot \nabla u = - c_s^2 \rho \nabla \rho - E - u \times B \\
\partial_t \rho + \nabla \cdot (\rho u) = 0 \\
\partial_t B + \nabla \times E = 0 \\
\partial_t E - \nabla \times B = \rho u \\
\nabla \cdot E = -n \\
\nabla \cdot B = 0\\
(u,n,E,B)(t=0) = (u_0,n_0,E_0,B_0).
\end{array}
\right.
\end{equation*}
We shall furthermore assume that, initially, 
\begin{equation}
\label{curlfree}
B=\nabla \times u.
\end{equation}
This condition is conserved by the flow of the above system: in order to see this, use the identity $u \cdot \nabla u = -u \times (\nabla \times u) + \nabla \frac{|u|^2}{2}$ to compute
\begin{equation*}
\begin{split}
\partial_t (B - \nabla \times u) & = \nabla \times \left( u \cdot \nabla u  + u \times B \right) \\
& = \nabla \times \left( - u \times (\nabla \times u) + \nabla \frac{|u|^2}{2} \right) - \nabla \times (u \times B) \\
& = \nabla \times \left( u \times (B - \nabla \times u) \right) .
\end{split}
\end{equation*}
The linearized system reads now
\begin{equation}
\label{oiseau}
\left\{ 
\begin{array}{l}
\left( \partial_t^2 - c_s^2 \Delta + 1 \right) \left( \begin{array}{c}Qu \\ n \\  QE \end{array} \right) = 0 \\
\left( \partial_t^2 - \Delta + 1 \right) \left( \begin{array}{c} Pu \\ PE \\ B \end{array} \right)= 0 .
\end{array}
\right.
\end{equation}

\subsection{Obtained results}

Prior to stating our theorem, we need to define the operator $A \overset{def}{=} \frac{\<D\>}{|D|}$ (see Section~\ref{notations} for the precise definition of this operator).

\begin{thm}
\label{mainthm}
Assume that the resonance separation condition~\ref{bluebird} holds; it is the case generically in $c_s$. Fix $\alpha_0>0$. Then there exists $C_0 , \epsilon_0 , N_0>0$ such that: if $\epsilon<\epsilon_0$, $N>N_0$ and
$$
\left\| \< x \>^{1+\alpha_0} (u_0,A n_0,E_0,A B_0) \right\|_{H^N} < \epsilon,
$$
then there exists a unique global solution of $(EM)$ such that
$$
\sup_t \left[ \< t \>^{-C_0 \epsilon} \| (u,A n,E,A B)(t) \|_{H^N} + \sqrt{\< t \>} \|(u,A n,E,A B)(t)\|_3 \right] \lesssim \epsilon.
$$
Furthermore, it scatters as $t $ goes to  infinity in that there exists a solution $(u_{\ell},n_\ell,E_\ell,B_\ell)$ of the linear system~(\ref{oiseau}) corresponding to intial data 
in $H^{N-2}$ such that
$$
\left\| (u,n,E,B)(t) - (u_{\ell},n_\ell,E_\ell,B_\ell)(t) \right\|_{H^{N-2}} \rightarrow 0 \;\;\;\;\;\;\;\;\mbox{as $t \rightarrow \infty$}.
$$
\end{thm}

\begin{rem} A few observations on the hypotheses on the initial data:
\begin{itemize}
\item
The requirements on $An_0$ and $AB_0$ imply necessarily that 
$\int n_0 =$  and $\int B_0 = 0$. In particular this is consistent with the electric neutrality.  
\item
We did not try to optimize the number of derivatives in $L^2$ required ($N$), but rather aimed at a proof as simple as possible. On the other hand, the weight appearing above ($\langle x\rangle^{1+\alpha}$) seems nearly optimal; a more precise analysis would maybe allow $\langle x \rangle$ instead of $\langle x \rangle^{1+\alpha}$.
\end{itemize}
\end{rem}
 
The proof will be essentially split into two parts: controling the $H^N$ norm of $(u,n,E,B)$; and proving the decay in various norms. The former is achieved by an energy estimate; and the latter by the method of space-time resonances, which was introduced in \cite{GMS09imrn}. It was also used to prove global existence of small data solutions for  water waves \cite{GMS09cras,GMS09ww}.

\subsection{Stability of compressible Euler and related models in dimension 3}

It is instructive to compare the above results to earlier works on compressible Euler in dimension 3, or couplings of compressible Euler with various fields (electrostatic, electromagnetic, gravitational...). For all these models, a fundamental question is whether given data lead to blow up or a global solution.

A first class of results gives blow up for various types of data. The fundamental work is due to Sideris~\cite{Sid85}, who proved finite time blow up for compressible Euler; he was able to obtain this result for data arbitrarily close to the equilibrium state given by a zero velocity, and a constant density. Many results followed: finite time blow up was showed for the compressible Euler equation with compactly supported data by Makino, Ukai, and Kawashima~\cite{MUK86}; for the attractive Euler-Poisson equation with compactly supported data by Perthame~\cite{Pert90}; for the repulsive Euler-Poisson equation with compactly supported data by Makino and Perthame~\cite{MP90}; and for the relativistic compressible Euler equation by Guo and Tahvildar-Zadeh~\cite{GT99} and Pan and Smoller~\cite{PS06}. 

All of the aforementioned results rely on a non-constructive proof, and do not say much about the nature of the singularity. Recently, Christodoulou~\cite{Chr07} was able to describe in a very precise manner the blow up process for the relativistic compressible Euler equation.

Another line of research gives global existence (and scattering) for data close to the equilibrium state given by constant density, and all the fields (including the velocity) equal to zero. Such of result was first obtained by Guo~\cite{Guo99} for repulsive Euler-Poisson; and by Guo and Pausader~\cite{GP11} for the ion dynamics in Euler-Maxwell. In both cases, the curl of the data is assumed to be zero, and this condition is conserved by the flow of the equation. Finally, global existence for Euler-Maxwell with relaxation was obtained by Duan~\cite{Dua10}.

Focusing on the case of small data (i.e. close to an equilibrium), some common features emerge from the results which have been mentioned. Global existence is only known under the assumption that the flow is irrotational: this eliminates a mode which is linearly non-decaying. Under this assumption, a crucial point is then the nature of the linearized equation: roughly speaking, blow up may occur if it is a wave equation, whereas global existence is expected if it is a Klein-Gordon equation. The relevant difference between these two situations is that the latter gives a decay $\sim \frac{1} {t^{3/2}}$, whereas the former only decays $\sim \frac{1}{t}$.

In the case of Euler-Maxwell, which is treated in this paper, the condition $B= \nabla \times u$ is also meant to restrict the solution to the subspace along which the linearized problem is governed by Klein-Gordon equations. The novelty is that these Klein-Gordon equations have different speeds, making the nonlinear interaction more intricate.

\section{Notations}

\label{notations}

We shall use the following standard notations:
\begin{itemize}
\item $A \lesssim B$ if $A \leq C B$ for some implicit constant $C$. The value of $C$ may change from line to line.
\item $A \sim B$ means that both $A \lesssim B$ and $B \lesssim A$.
\item For any real number $\alpha$, the ``japanese brackets''  $\langle \cdot \rangle_{\alpha}$ stand for $\langle x \rangle_{\alpha} = \sqrt{1+\alpha^2 x^2}$. We also denote $ \langle x \rangle =\langle x \rangle_{1}$.
\item If $f$ is a function over $\mathbb{R}^3$ then its Fourier transform, denoted  by $\widehat{f}$, or $\mathcal{F} f$, is given by
$$
\widehat{f}(\xi) = \mathcal{F}f (\xi) = \frac{1}{(2\pi)^{3/2}} \int e^{-ix\xi} f(x) \,dx \;\;\;\;\mbox{thus} \;\;\;\;f(x) = \frac{1}{(2\pi)^{3/2}} \int e^{ix\xi} \widehat{f}(\xi) \,d\xi.
$$
In the text, we systematically drop the constants such as $\frac{1}{(2 \pi)^{3/2}}$ since they are not relevant.
\item The Fourier multiplier with symbol $m(\xi)$ is defined by
$$
m(D)f = \mathcal{F}^{-1} \left[m \mathcal{F} f \right].
$$
\item The bilinear pseudo-product with symbol $m(\xi,\eta)$ is given by its Fourier transform
$$
\mathcal{F}\left[ T_m(f,g) \right] (\xi) = \int m(\xi,\eta) \widehat{f}(\eta) \widehat{g}(\xi-\eta) \,d\eta.
$$
Similarly, the trilinear pseudo-product with symbol $m(\xi,\eta,\nu)$ is given by
$$
\mathcal{F}\left[ T_m(f,g,h) \right] (\xi) = \int m(\xi,\eta,\nu) \widehat{f}(\nu) \widehat{g}(\eta) \widehat{h}(\xi-\eta-\nu) \,d\eta\,d\nu.
$$
\item $H^N$ is given by the norm $\|f\|_{H^N} = \| \<D\>^N f \|_2$.
\item $W^{s,p}$ is given by the norm $\|f\|_{W^{s,p}} = \| \<D\>^{s} f \|_p$.
\end{itemize}

\section{A formulation adapted to energy estimates}

Our aim here is to rewrite the equation in such a way that its dispersive properties become more transparent, but energy estimates can also be easily obtained.

Split
$$
\left( \begin{array}{l} u \\ n \\ E \\ B \end{array} \right) = \left( \begin{array}{l} Qu \\ n \\ QE \\ 0 \end{array} \right) + \left( \begin{array}{l} Pu \\ 0 \\ PE \\ B \end{array} \right) \overset{def}{=} V_a + V_p
$$
where $V_p$ contains the unknowns which (in the linearization~(\ref{oiseau})) propagate as electromagnetic waves, and $V_a$ the unknowns which (still in the linearization~(\ref{oiseau})) propagate as acoustic waves.

\subsection{The acoustic system}

We focus here on the evolution of $V_a = \left( Qu \,,\, n \,,\, QE \,,\, 0 \right)$. It is governed by the system
$$
\left\{ \begin{array}{l} \partial_t Qu = -QE - \nabla \frac{|u|^2}{2} - c_s^2 \rho \nabla \rho \\ \partial_t n = - \nabla \cdot (\rho u) \\ \nabla \cdot E = -n. \end{array} \right.
$$
In order to diagonalize this system, let us switch to the unknown function
$$
\mathcal{A} = \frac{1}{2} \left( \frac{\langle D \rangle_{c_s}}{|D|} n +  i \frac{\nabla}{|D|} \cdot u \right)
$$
so that
$$
Qu = - 2 \frac{\nabla}{|D|} \frak{Im} \mathcal{A} \;\;\;\;\;\mbox{and}\;\;\;\;\; n = 2 \frac{|D|}{\langle D \rangle_{c_s}} \frak{Re} \mathcal{A}.
$$
The evolution of $\mathcal{A}$ is given by
\begin{equation}\label{Aequat}
 2 \partial_t \mathcal{A} = 2  i \langle D \rangle_{c_s} \mathcal{A} - 
 \frac{\langle D \rangle_{c_s} \nabla}{|D|} \cdot (n u)  + \frac{i |D|}2  \left( \left| u \right|^2 + c_s^2 |n|^2 \right).
\end{equation}

\subsection{The Maxwell (or electromagnetic) system}

We focus here on the evolution of $V_p = \left( Pu \,,\, 0 \,,\, PE \,,\, B \right)$. By (\ref{curlfree}), it suffices to consider $PE$ and $B$. These fields are governed by the equations
\begin{equation*}
\left\{ \begin{array}{l} \partial_t B = - \nabla \times E \\ \partial_t PE = \nabla \times B + P(\rho u) \end{array} \right.
\end{equation*}
which implies
$$
\partial_t^2 B - \Delta B + B = - \nabla \times (n u) .
$$
Setting
$$
\mathcal{B} = \frac{\partial_t}{|D|} B + i \frac{\<D\>}{|D|} B,
$$
it satisfies
$$
\partial_t \mathcal{B} - i \langle D \rangle \mathcal{B} = - \frac{\nabla}{|D|} \times (n u),
$$
and the original unknown functions $Pu$, $PE$ and $B$ can be recovered by
$$
Pu = \frac{\nabla}{|D|\<D\>} \times \frak{Im} \mathcal{B} \;\;\;\;,\;\;\;\;
PE = -\frac{\nabla}{|D|} \times \frak{Re} \mathcal{B} \;\;\;\;\;\;\mbox{and}\;\;\;\;\;\; B = \frac{|D|}{\<D\>} \frak{Im} \mathcal{B}.
$$

\subsection{Summarizing}

The Euler-Maxwell system now reads
$$
(EM') \quad \left\{ \begin{array}{l} \partial_t \mathcal{A} - i \langle D \rangle_{c_s} \mathcal{A} = -\frac{1}{2} \frac{\langle D \rangle \nabla}{|D|} \cdot (n u) + \frac{1}{4} i |D| \left( \left| u \right|^2 + |n|^2 \right) \\ \partial_t \mathcal{B} - i \langle D \rangle \mathcal{B} = - \frac{\nabla}{| D |} \times (n u) \\ (\mathcal{A},\mathcal{B})(t=0)=(\mathcal{A}_0,\mathcal{B}_0) \end{array} \right.
$$
with
$$
\left\{ \begin{array}{l} Qu = - 2 \frac{\nabla}{|D|} \frak{Im} \mathcal{A} \\ n = 2 \frac{|D|}{\langle D \rangle_{c_s} } \frak{Re} \mathcal{A} \\ Pu = \frac{\nabla}{|D|\<D\>} \times \frak{Im} \mathcal{B}. \end{array} \right.
$$
The data $(\mathcal{A}_0,\mathcal{B}_0)$ of $(EM')$ are easily expressed in terms of the data $(u_0,n_0,E_0,B_0)$ of $(EM)$:
$$
\mathcal{A}_0 \overset{def}{=} \frac{1}{2} \left( \frac{\langle D \rangle_{c_s}}{|D|} n_0 +  i \frac{\nabla}{|D|} \cdot u_0 \right) \;\;\;\;\mbox{and} \;\;\;\;\mathcal{B}_0 = - \frac{\nabla}{|D|} \times E_0 + i \frac{\<D\>}{|D|} B_0.
$$
Let us finally define the profiles of $\mathcal{A}$ and $\mathcal{B}$
$$
a(t) \overset{def}{=} e^{-it\<D\>_{c_s}} \mathcal{A}(t) \;\;\;\;\mbox{and}\;\;\;\; b(t) \overset{def}{=} e^{-it\<D\>} \mathcal{B}(t).
$$

\section{A formulation adapted to decay estimates}

As we saw, the system $(EM')$ written above is equivalent to $(EM)$; it will be the correct formulation to perform energy estimates. However, as far as dispersive estimates go, we will not need all the structure of $(EM')$: only resonances will play an important r\^ole. It will be convenient to write $(EM')$ in a more compact form.

\subsection{Duhamel's formula in Fourier space}

\label{duhamel}

Writing Duhamel's formula in terms of $a$ and $b$ gives
$$
\left\{ \begin{array}{l} a(t) = \mathcal{A}_0 + \int_0^t e^{-is \langle D \rangle_{c_s}} \left[ -\frac{1}{2} \frac{\langle D \rangle \nabla}{|D|} \cdot (n u) + \frac{1}{4} i |D| \left( \left| u \right|^2 + |n|^2 \right) \right] \,ds \\
b(t) = \mathcal{B}_0 - \int_0^t e^{-is \langle D \rangle} \left[\frac{\nabla}{| D |} \times (n u)\right]\,ds.
\end{array} \right.
$$
Taking the Fourier transform gives
$$
\left\{ \begin{array}{l}
\widehat{a}(t,\xi) = \widehat{\mathcal{A}_0}(\xi) + \mbox{``nonlinear term''} \\
\widehat{b}(t,\xi) = \widehat{\mathcal{B}_0}(\xi) + \mbox{``nonlinear term''}.
\end{array} \right.
$$
In order to make notations lighter and estimates easier, we will now give up some of the structure of the above system. 

\begin{conv}
We will denote indifferently $\mathcal{C}(t)$ for $\mathcal{A}(t)$ or $\mathcal{B}(t)$, or their complex conjugates, and $c(t)$ for $a(t)$ or $b(t)$, or their complex conjugates. Similarly, we denote $e^{\pm it \langle D \rangle_\ell}$ for any of the groups $e^{it \langle D \rangle}$, $e^{-it \langle D \rangle}$, $e^{it \langle D \rangle_{c_s}}$, or $e^{-it \langle D \rangle_{c_s}}$.

For instance, $u$ or $n$ are a linear combination of terms of the type $e^{\pm it \langle D \rangle_\ell} c$.

It is always understood that, in an expression of the form $e^{\pm it \langle D \rangle_\ell} c(t)$, the meaning of $e^{\pm it \langle D \rangle_\ell}$ is consistent with that of $c$. For instance, if $c$ stands for $a$, then $e^{\pm it \langle D \rangle_\ell}$ stands for $e^{it\langle D \rangle_{c_s}}$.
\end{conv}

With this convention, it is easy to see from the above that the ``nonlinear terms'' can all be written as a linear combination of terms of the following type (which we denote generically by $g$)
\begin{equation}
\label{pinguin}
\widehat{g}(t,\xi) = \int_0^t  \int e^{is\phi(\xi,\eta)} m(\eta,\xi) \widehat{c}(\eta,s) \widehat{c}(\xi-\eta,s) \,  
 d\eta  \, ds,
\end{equation}
where $m$ is such that
\begin{equation}
\label{eider}
m(\xi,\eta) = m_0(\xi) m_1(\eta) m_2(\xi - \eta) \quad\quad \mbox{with} \left\{ \begin{array}{l} \left|\partial_\xi^\alpha m_0(\xi) \right| \lesssim \frac{1}{|\xi|^{|\alpha|}} \quad \mbox{if $|\xi| \leq 1$} \\ \left|\partial_\xi^\alpha m_0(\xi) \right| \lesssim \frac{1}{|\xi|^{|\alpha|-1}} \quad \mbox{if $|\xi| \geq 1$} \\ \left|\partial_\xi^\alpha m_1(\xi) \right|,\;\left|\partial_\xi^\alpha m_2(\xi) \right| \lesssim \frac{1}{|\xi|^{|\alpha|}} \quad \mbox{for any $\xi$} \end{array} \right.
\end{equation}
and $\phi$ is one of the $\phi^{\epsilon_1, \epsilon_2}_{k,\ell,m}$ given by
$$
\phi^{\epsilon_1, \epsilon_2}_{k,\ell,m}(\xi,\eta) \overset{def}{=} \langle \xi \rangle_{k} + \epsilon_1 \langle \eta \rangle_{\ell} + \epsilon_2 \langle \xi - \eta \rangle_{m}
$$
where $\epsilon_1, \epsilon_2 = \pm$ and $k,\ell,m$ are either $1$ or $c_s$.

\subsection{Space-time resonances in the context of Euler-Maxwell}

\label{wigeon}

Seeing~(\ref{pinguin}) as an oscillatory integral, it becomes clear that the cancellation properties of $\phi$ and $\partial_\eta \phi$ will provide a key to understanding the large time behaviour of our system: this is the idea of space time resonances. See~\cite{PG1} for a general presentation, and~\cite{PG} for the case of (semilinear) Klein-Gordon equations with different propagation speeds.

Recall that the phase functions corresponding to all possible quadratic interactions are given by
\begin{equation} \label{phase}
\phi^{\epsilon_1, \epsilon_2}_{k,\ell,m}(\xi,\eta) \overset{def}{=} \langle \xi \rangle_{k} + \epsilon_1 \langle \eta \rangle_{\ell} + \epsilon_2 \langle \xi - \eta \rangle_{m}
\end{equation}
Next define for each interaction the space, time, and space-time resonant sets
\begin{equation*}
\begin{split}
& \mathcal{S}^{\epsilon_1, \epsilon_2}_{k,\ell,m} \overset{def}{=} \{  (\xi, \eta) \, |\,  \phi^{\epsilon_1, \epsilon_2}_{k,\ell,m} = 0 \} \;\;\;\;\;\;\mbox{(''space resonances")}\\ 
& \mathcal{T}^{\epsilon_1, \epsilon_2}_{k,\ell,m} \overset{def}{=} \{   (\xi, \eta)  \, | \,  \partial_\eta \phi^{\epsilon_1, \epsilon_2}_{k,\ell,m} = 0 \} \;\;\;\;\;\;\mbox{(''time resonances")}\\
& \mathcal{R}^{\epsilon_1, \epsilon_2}_{k,\ell,m} \overset{def}{=} \mathcal{S}^{\epsilon_1, \epsilon_2}_{k,\ell,m} \cap \mathcal{T}^{\epsilon_1, \epsilon_2}_{k,\ell,m}. \;\;\;\;\;\;\mbox{(''space-time resonances")}
\end{split}
\end{equation*}
The set of all space-time resonances is
$$
\mathcal{R} = \cup_{\epsilon_1, \epsilon_2,k,\ell,m} \mathcal{R}^{\epsilon_1, \epsilon_2}_{k,\ell,m};
$$
it is compact and hence it is bounded. We denote by $ C_\mathcal{R} -1  $ 
the radius of a ball that contains $\mathcal{R}  $.
Finally, define the outcome and germ, or source frequencies for space-time resonances: these are simply the projections of $\mathcal{R}^{\epsilon_1, \epsilon_2}_{k,\ell,m}$ in the $\xi$ variable, respectively the union of the projections in the $\eta$ and $\xi-\eta$ variables. More precisely if $\pi_\xi(\xi',\eta') = \xi'$, $\pi_\eta(\xi',\eta') = \eta'$ and $\pi_{\xi-\eta}(\xi',\eta') = \xi'-\eta'$, we set
\begin{equation*}
\begin{split}
& \mathcal{O} \overset{def}{=} \pi_\xi \left( \mathcal{R} \right) \\
& \mathcal{G} \overset{def}{=} \pi_\eta \left( \mathcal{R}\right) \cup \pi_{\xi-\eta} \left( \mathcal{R} \right).
\end{split}
\end{equation*}

\begin{df} 
\label{bluebird}
Space-time resonances are said to be separated if no outcome frequency is also a germ frequency. In mathematical terms,
$ \displaystyle \mathcal{G} \cap \mathcal{O} = \emptyset $.
\end{df}

\section{Some linear and bilinear cutoff Fourier multipliers}

\label{sectionmachinery}

We overtake here some of the cut-off functions defined in~\cite{PG}; see Proposition~\ref{toucan} for results on the boundedness of the associated operators.

%
%

\subsection{Low or high frequency cutoff: $Z_l$, $Z_h$}

First pick $M_0$ large enough (the precise value of $M_0$ will be fixed in the following, for the moment it is simply $\geq C_\mathcal{R}$ defined in Section~\ref{wigeon}).

It will be necessary in the proof to distinguish between high and low frequencies. To this end, we introduce the cut off function $\theta(\xi,\eta)$, which is such that
\begin{equation}
\label{colvert}
\theta \in \mathcal{C}^\infty_0\;\;\;\;,\;\;\;\;\theta = 1 \;\mbox{on $B(0,1)$}\;\;\;\;\mbox{and}\;\;\;\;\theta = 0 \;\mbox{on $B(0,2)^c$}.
\end{equation}
Restricting to high, respectively low frequencies, is achieved by the operators
$$
Z_h \overset{def}{=} 1-\theta\left( \frac{D}{M_0} \right) \quad\quad Z_l \overset{def}{=} \theta \left( \frac{D}{M_0} \right) .
$$

\subsection{Cutoff for $\mathcal{O}$: the operators $Z_\mathcal{O}$, $\widetilde{Z}_{\mathcal{O}}$}

Recall that $\mathcal{O}$ and $\mathcal{G}$ were defined in Section~\ref{wigeon}.

Under the resonance separation condition (definition~\ref{bluebird}), it is possible to find $\delta_0$ such that no frequency in $B_{10\delta_0}(\mathcal{O})$ (a $10\delta_0$-neighbourhood of $\mathcal{O}$) is a source of a space-time resonance:
$$
B_{10\delta_0}(\mathcal{O}) \cup \mathcal{G} = \emptyset.
$$
Define $\chi_\mathcal{O}$ a smooth cut-off function such that 
\begin{equation*}
\begin{split}
& \chi_\mathcal{O} = 1 \;\;\;\mbox{on $B_{\delta_0/2}(\mathcal{O})$} \\
& \chi_\mathcal{O} = 0 \;\;\;\mbox{outside of $B_{\delta_0}(\mathcal{O})$}
\end{split}
\end{equation*}
and then let $\widetilde{\chi}_{\mathcal{O}}$ satisfy
$$
\chi_\mathcal{O} + \widetilde{\chi}_{\mathcal{O}} = 1.
$$
The corresponding operators are
$$
Z_\mathcal{O} \overset{def}{=} \chi_\mathcal{O}(D) \;\;\;\;\;\mbox{and}\;\;\;\;\; \widetilde{Z}_\mathcal{O} \overset{def}{=} \widetilde{\chi}_\mathcal{O}(D).
$$

\subsection{Cutoff for $\mathcal{S}$ and $\mathcal{T}$: the symbols $\chi_\mathcal{S}$ and $\chi_\mathcal{T}$}

The cut-off functions which we are about to define will, for a given set of indices $\epsilon_1,\epsilon_2,k,\ell,m$ separate $\mathcal{T}^{\epsilon_1,\epsilon_2}_{k,\ell,m}$ from $\mathcal{S}^{\epsilon_1,\epsilon_2}_{k,\ell,m}$; of course this can only be done away from a neighbourhood of $\mathcal{R}^{\epsilon_1,\epsilon_2}_{k,\ell,m}$, where these two sets intersect. Dropping for simplicity the indices, the function $\chi_\mathcal{S}$ localizes away from $\mathcal{T}$, whereas $\chi_\mathcal{T}$ localizes away from $\mathcal{S}$. Since $\mathcal{T}=\{ \phi = 0 \}$ whereas $\mathcal{S} = \{ \partial_\eta \phi = 0 \}$, this explains the inequalities~(\ref{rossignol2}).

\begin{lem}
\label{aigle2}
For each set of indices $\epsilon_1,\epsilon_2, k,\ell,m$, it is possible to find cut-off functions
$$
\chi_{\mathcal{T}^{\epsilon_1,\epsilon_2}_{k,\ell,m}}(\xi,\eta) \;\;,\;\;\chi_{\mathcal{S}^{\epsilon_1,\epsilon_2}_{k,\ell,m}}( \xi,\eta)
$$
such that (in the following, we drop the indices $\epsilon_1,\epsilon_2 , k,\ell,m$ for simplicity)
\begin{itemize}
\item $\chi_\mathcal{T}$, $\chi_\mathcal{S}$ are smooth.
\item Their sum equals one away from $\mathcal{R}$: $\chi_{\mathcal{T}} + \chi_{\mathcal{S}} = 1$ if $\operatorname{dist}((\xi,\eta),\mathcal{R})>\delta_0/10$.
\item The derivatives of $\frac{\chi_{\mathcal{S}}}{\phi}$ and $\frac{\chi_{\mathcal{T}} \partial_\eta \phi}{|\partial_\eta \phi|^2}$ satisfy
\begin{equation}
\label{rossignol2}
\mbox{ if $|\alpha| \leq 20$,  then } \;\;\;\;
\left| \partial_{\xi,\eta}^{\alpha} \frac{\chi_{\mathcal{S}}}{\phi} \right| \;,\;
\left| \partial_{\xi,\eta}^{\alpha} \frac{\chi_{\mathcal{T}} \partial_\eta \phi}{|\partial_\eta \phi|^2} \right| \lesssim |\xi,\eta|^{n_0}
\end{equation}
for some  integer $n_0$.
\end{itemize}
\end{lem}

\subsection{Paraproduct decomposition: the symbols $\zeta^1$ and $\zeta^2$}

Following the original idea of Bony~\cite{B}, we would like to distinguish between regions where $|\eta| \gtrsim |\xi-\eta|$, respectively $|\xi-\eta|\gtrsim |\eta|$. 

We pick two functions $\zeta^1(\xi,\eta)$ and $\zeta^2(\xi,\eta)$ such that
\begin{itemize}
\item $\zeta^2$ and $\zeta^1$ are smooth.
\item $\zeta^2$ and $\zeta^1$ are homogeneous of degree zero outside of $B(0,1)$.
\item $\zeta^2(\xi,\eta) + \zeta^1(\xi,\eta) = 1$ for any $(\xi,\eta)$.
\item If $|(\xi,\eta)| \geq 1$ and $(\xi,\eta) \in \operatorname{Supp}(\zeta^1)$, then $|\xi-\eta|\leq c |\eta|$ for a constant $c$.
\item If $|(\xi,\eta)| \geq 1$ and $(\xi,\eta) \in \operatorname{Supp}(\zeta^2)$, then $|\eta |\leq c |\xi-\eta|$ for a constant $c$.
\end{itemize}

\section{The a priori estimates and plan of the proof}

In order to prove Theorem~\ref{mainthm}, we will prove the following a priori estimates, valid if $\epsilon$ is small enough.

\bigskip
\noindent
\underline{Energy estimate}
\begin{itemize}
\item $\displaystyle \left\| (\mathcal{A},\mathcal{B}) \right\|_{H^N} \lesssim \epsilon  \<t\>^{C_0 \epsilon}$ for a constant $C_0$, and any $t$ (regularity in $L^2$).
\end{itemize}
\underline{Decay estimates}
\begin{itemize}
\item $\displaystyle \left\| (\mathcal{A},\mathcal{B}) \right\|_{W^{N'',\left(\frac{1}{3}-\delta_1 \right)^{-1}}} \lesssim \frac{\epsilon}{\<t\>^{\frac{1}{2}+3\delta_1}}$ (square integrable decay above $L^3$).
\item $\displaystyle \left\| \widetilde{Z}_{\mathcal{O}} \left( \mathcal{A},\mathcal{B} \right) \right\|_{W^{N'',\left( \frac{1}{6}+\delta_1 \right)^{-1}}} \lesssim \frac{\epsilon}{\<t\>^{1 -3 \delta_1}}$ (decay slightly below $L^6$ for ``non-outcome'' frequencies).
\item $\displaystyle \left\| \widetilde{Z}_{\mathcal{O}} \left( \mathcal{A},\mathcal{B} \right) \right\|_{W^{2,\infty}} , \left\| \widetilde{Z}_{\mathcal{O}} \left( u,n \right) \right\|_{W^{2,\infty}} \lesssim \frac{\epsilon}{\<t\>}$ (decay $\sim \frac{1}{t}$ in $L^\infty$ for ``non-outcome'' frequencies).
\end{itemize}
\underline{Localization estimates}
\begin{itemize}
\item $\displaystyle \left\| |x| (a,b) \right\|_{H^{N'}} \lesssim \epsilon \sqrt{\<t\>}$ (localization in $H^{N'}$)
\item $\displaystyle \left\| |x|^{1/8} \widetilde{Z}_{\mathcal{O}} (a,b) \right\|_2 \lesssim \epsilon$ (localization in $L^2$ for ``non-outcome'' frequencies).
\end{itemize}
The constants $N$, $N'$, $N''$ are chosen such that $N - N_1> N'' - N_1 > N'  >N_1$, for a sufficiently big constant $N_1$; in particular, $N$ is sufficiently big for the necessary arguments in~\cite{PG} to apply. The constant $\delta_1$ is chosen sufficiently small for the necessary parts of the argument in~\cite{PG} to apply. We call $\|\cdot\|_X$ the norm corresponding to the above quantities:
\begin{equation*}
\begin{split}
\|(\mathcal{A},\mathcal{B})\|_X & \overset{def}{=} \sup_t \left[    \<t\>^{-C_0 \epsilon} \left\| (\mathcal{A},\mathcal{B}) \right\|_{H^N} + \<t\>^{\frac{1}{2}+3\delta_1} \left\| (\mathcal{A},\mathcal{B}) \right\|_{W^{N'',\left(\frac{1}{3}-\delta_1\right)^{-1}}} + \<t\>^{1-3\delta_1} \left\| (\mathcal{A},\mathcal{B}) \right\|_{W^{N'',\left( \frac{1}{6}+\delta_1 \right)^{-1}}} \right.\\ 
& + \left. \<t\> \left\| \widetilde{Z}_{\mathcal{O}} (\mathcal{A},\mathcal{B},u,n) \right\|_{W^{2,\infty}} + \frac{1}{\sqrt{\<t\>}} \left\| |x| (a,b) \right\|_{H^{N'}} + \left\| |x|^{1/8} \widetilde{Z}_{\mathcal{O}} (a,b) \right\|_2 \right]
\end{split}
\end{equation*}

Since local well posedness is easily dealt with, and the data are chosen such that 
$$
\|(e^{it\<D\>_{c_s}} \mathcal{A}_0 , e^{it\<D\>} \mathcal{B}_0)\|_X \lesssim \epsilon,
$$
the proof of the theorem consists in proving the following a priori estimate:
$$
\|(\mathcal{A},\mathcal{B})\|_X \lesssim \|(e^{it\<D\>_{c_s}} \mathcal{A}_0 , e^{it\<D\>} \mathcal{B}_0)\|_X + \|(\mathcal{A},\mathcal{B})\|_X^2 + \|(\mathcal{A},\mathcal{B})\|_X^3.
$$
We will proceed by showing that all the quantities appearing in the definition of $X$ can be controlled by the above right-hand side. More precisely, the plan is as follows
\begin{itemize}
\item Decay estimates are proved in Section~\ref{sectiondecay}.
\item Localization estimates are proved in Section~\ref{sectionlocalization}.
\item The energy estimate for $\mathcal{A}$: $\sup_t   \<t\>^{-C_0 \epsilon}  \| \mathcal{A}(t) \|_{H^N} \lesssim \epsilon$ is proved in Section~\ref{acoustic}.
\item The energy estimate for $\mathcal{B}$: $\sup_t    \<t\>^{-C_0 \epsilon}  \| \mathcal{B}(t) \|_{H^N} \lesssim \epsilon$ is proved in Section~\ref{plasma}.
\item Finally, in section~\ref{scat}  we give a sketch of the proof of the scattering. 
\end{itemize}

\section{Decay estimates}

\label{sectiondecay}

We want to prove here that
\begin{equation}
\label{swan}
\begin{split}
& \sup_t \left[ \<t\>^{\frac{1}{2}+3\delta_1} \left\| (\mathcal{A},\mathcal{B}) \right\|_{W^{N'',\left(\frac{1}{3}-\delta_1\right)^{-1}}} + \<t\>^{1-3\delta_1} \left\| \widetilde{Z}_\mathcal{O} (\mathcal{A},\mathcal{B}) \right\|_{W^{N'',\left(\frac{1}{6}+\delta_1\right)^{-1}}} + 
\<t\> \left\| \widetilde{Z}_\mathcal{O} (\mathcal{A},\mathcal{B},u,n) \right\|_\infty \right] \\
& \quad \quad \quad \quad \quad \quad \lesssim \left\| (e^{it\<D\>_{c_s}}\mathcal{A}_0 + e^{it\<D\>}\mathcal{B}_0 ) \right\|_X + \left\|(\mathcal{A},\mathcal{B} )\right\|_X^2.
\end{split}
\end{equation}

\subsection{Control of the $W^{N'',\left(\frac{1}{6}+\delta_1\right)^{-1}}$ and $W^{N'',\left(\frac{1}{3}-\delta_1\right)^{-1}}$ norms}

The two first norms in~(\ref{swan}) above can be controlled as in~\cite{PG}:
\begin{itemize}
\item As far as the control of the $W^{N'',\left(\frac{1}{6}+\delta_1\right)^{-1}}$ norm goes, the main difference between the Euler-Maxwell system and the setting of~\cite{PG} is the quasilinearity of Euler-Maxwell. This induces a further loss of regularity in the nonlinear term, which is however easily absorbed using the $H^N$ norm.
\item The estimate for the $W^{N'',\left(\frac{1}{3}-\delta_1\right)^{-1}}$ norm is a low frequency question (since it is only problematic on $\mathcal{O}$). Therefore, the argument of~\cite{PG} applies identically.
\end{itemize}
We do not detail these two points, and focus directly on the third norm in~(\ref{swan}). 

\subsection{Control of the $W^{2,\infty}$ norm}

Proceeding as in Subsection~\ref{duhamel}, we can derive a generic term $g$ corresponding to the nonlinear term in Duhamel's formula for $u$ and $n$. It turns out, since $u$ and $n$ are given from $\mathcal{A}$ and $\mathcal{B}$ by the action of a Fourier multiplier, that this $g$ would satisfy exactly the properties listed in Subsection~\ref{duhamel}.

Thus all we need to do is to prove that, for $g$ as in Subsection~\ref{duhamel},
$$
\left\| \widetilde{Z}_\mathcal{O} e^{it\<D\>_\ell} g(t) \right\|_{W^{2,\infty}} \lesssim \frac{1}{\<t\>} \left\|(\mathcal{A},\mathcal{B})\right\|_X^2.
$$
In order to prove this, we shall split $\widetilde{Z}_\mathcal{O} e^{it\<D\>_\ell} g(t)$ as follows
\begin{subequations}
\begin{align} \label{eagle1}
\mathcal{F} ( \widetilde{Z}_\mathcal{O} e^{it\<D\>_\ell} g(t) ) = & \int_0^1 \int \widetilde{\chi}_\mathcal{O}(\xi)  e^{is\phi} m(\xi,\eta) \widehat{c}(s,\eta) \widehat{c}(s,\xi-\eta) \,d\eta\,ds \\
\label{eagle2}
& + \int_1^t \int \widetilde{\chi}_\mathcal{O}(\xi) \chi_\mathcal{S}(\xi,\eta) e^{is\phi} m(\xi,\eta) \widehat{c}(s,\eta) \widehat{c}(s,\xi-\eta) \,d\eta\,ds \\
\label{eagle3}
& + \int_1^t \int \widetilde{\chi}_\mathcal{O}(\xi) \chi_\mathcal{T}(\xi,\eta) e^{is\phi} m(\xi,\eta) \widehat{c}(s,\eta) \widehat{c}(s,\xi-\eta) \,d\eta\,ds
\end{align}
\end{subequations}

In the above, we have used the cut-off functions $\chi_\mathcal{S}$ and $\chi_\mathcal{T}$. Remember that these were defined in~\ref{aigle2} depending on the quadratic interaction considered; they were therefore labeled $\chi_{\mathcal{S}^{k,l,m}_{\epsilon_1,\epsilon_2}}$ and $\chi_{\mathcal{T}^{k,l,m}_{\epsilon_1,\epsilon_2}}$. The above equation is written in generic terms, but it is tacitly understood that the cut-off functions used are the ones corresponding to the quadratic interaction at hand.

\subsection{Preliminary estimate on $\partial_s c$}

Observe from subsection~\ref{duhamel} that $e^{is\<\xi\>_k} \partial_s \widehat{c}(\xi)$ can be written as a sum of terms of the type
$$
\int m(\xi,\eta) \widehat{\mathcal{C}}(\eta) \widehat{\mathcal{C}}(\xi-\eta)\,ds,
$$
where $m$ satisfies the estimates of that section. Therefore, by proposition~\ref{toucan},
\begin{equation}
\label{flamingo}
\left\| e^{is\<D\>} \partial_s c \right\|_{W^{N''-1,3/2}} \lesssim \frac{1}{t} \left\| c \right\|_{W^{N'',3}}^2
\end{equation}

\subsection{The small time term~( \ref{eagle1} ) }

Using repeatedly the Sobolev embedding theorem, and the dispersive estimate~(\ref{crane}) gives (assuming $t>1$, the case $t<1$ being trivial)
\begin{equation*}
\begin{split}
\left\|  e^{it\<D\>_\ell} \mathcal{F}^{-1} (\ref{eagle1}) \right\|_{W^{2,\infty}} = & \left\| \int_0^1 e^{i(t-s)\<D\>_\ell}  T_{\widetilde{\chi}_\mathcal{O}(\xi) m(\xi,\eta)} (\mathcal{C},\mathcal{C})\,ds \right\|_{W^{3,6}} \\
& \lesssim \frac{1}{t} \left\| \int_0^1 T_{\widetilde{\chi}_\mathcal{O}(\xi) m(\xi,\eta)} (\mathcal{C},\mathcal{C})\,ds \right\|_{W^{5,6/5}} \\
& \lesssim \frac{1}{t} \int_0^1 \left\| \mathcal{C} \right\|_{W^{6,12/5}}^2 \,ds \\
& \lesssim \frac{1}{t} \int_0^1 \left\| \mathcal{C} \right\|_{H^7}^2 \,ds \lesssim \frac{1}{t} \left\| \mathcal{C} \right\|_X^2.
\end{split}
\end{equation*}

\subsection{The term away from $\mathcal{T}$~(\ref{eagle2})}

In order to deal with this term, integrate by parts in time using the identity $\frac{1}{i\phi} \partial_s e^{is\phi} = e^{is\phi}$. Thus
\begin{subequations}
\begin{align}
\label{swallow1}
(\ref{eagle2}) = & \int \widetilde{\chi}_\mathcal{O}(\xi) \chi_\mathcal{S}(\xi,\eta)  m(\xi,\eta) 
\frac{1}{i\phi} \widehat{\mathcal{C}} (t,\eta)\widehat{\mathcal{C}} (t,\xi-\eta) \, d\eta\\
\label{swallow2}
& - \int_1^t \int \widetilde{\chi}_\mathcal{O}(\xi) \chi_\mathcal{S}(\xi,\eta)  m(\xi,\eta) 
\frac{1}{i\phi} e^{i s \phi}  \partial_s \widehat{c} (s,\eta) \widehat{c} (s,\xi-\eta) \, d\eta\,ds \\
& \;\;\;\;\;\;\;\;\;\;\;\;\;\;\;\;\;\;\;\;\;\;\mbox{+ \{symmetric and easier terms\}}, 
\end{align}
\end{subequations}
where the ``symmetric and easier terms'' correspond to the case where the partial derivative $\partial_s$ hits the other $c$, and to the boundary term at $s=1$. Using successively the Sobolev embedding theorem~\ref{sobolev} and Proposition~\ref{toucan} gives
\begin{equation*}
\begin{split}
\left\| e^{it\<D\>_\ell} \mathcal{F}^{-1} (\ref{swallow1}) \right\|_{W^{2,\infty}} & = \left\|T_{\widetilde{\chi}_\mathcal{O}(\xi) \chi_\mathcal{S}(\xi,\eta) m(\xi,\eta) \frac{1}{i\phi}}(\mathcal{C},\mathcal{C}) \right\|_{W^{2,\infty}} \\
& \lesssim \left\| T_{\widetilde{\chi}_\mathcal{O}(\xi) \chi_\mathcal{S}(\xi,\eta) m(\xi,\eta) \frac{1}{i\phi}} (\mathcal{C},\mathcal{C}) \right\|_{W^{4,\left( \frac{2}{3} - 2 \delta_1 \right)^{-1}}} \\
& \lesssim \|\mathcal{C}\|_{W^{n+4,\left( \frac{1}{3} - \delta_1 \right)^{-1}}} \|\mathcal{C}\|_{W^{n+4,\left( \frac{1}{3} - \delta_1 \right)^{-1}}} \\
& \lesssim \frac{1}{t^{1+6\delta_1}} \|\mathcal{C}\|_X^2.
\end{split}
\end{equation*}
In order to estimate~(\ref{swallow2}), split it as follows
\begin{equation*}
\begin{split}
\mathcal{F}^{-1} (\ref{swallow2}) & = \int_1^t \int \widetilde{\chi}_\mathcal{O}(\xi) \chi_\mathcal{S}(\xi,\eta)  m(\xi,\eta) 
\frac{1}{i\phi} e^{i s \phi}  \partial_s \widehat{c} (s,\eta) \widehat{c} (s,\xi-\eta) \, d\eta\,ds \\
& = \int_1^{t-1} + \int_{t-1}^t \dots \overset{def}{=} I + II.
\end{split}
\end{equation*}
Use the Sobolev embedding theorem, the dispersive estimate~(\ref{crane}), Proposition~\ref{toucan} and the preliminary estimate~(\ref{flamingo}) to get, for $\delta>0$ small enough,
\begin{equation*}
\begin{split}
& \left\| e^{it\<D\>_\ell} \mathcal{F}^{-1} I \right\|_{W^{2,\infty}} \lesssim \int_1^{t-1} \frac{1}{(t-s)^{(3/2-3\delta)}} \left\| T_{\widetilde{\chi}_\mathcal{O}(\xi) \chi_\mathcal{S}(\xi,\eta) m(\xi,\eta) \frac{1}{i\phi}} \left(e^{\pm is\<D\>_\ell} (\partial_s c),\mathcal{C}\right) \right\|_{W^{5,(1-\delta)^{-1}}}\,ds \\
& \;\;\;\;\;\;\;\;\;\;\;\;\; \lesssim \int_1^{t-1} \frac{1}{(t-s)^{3/2-3\delta}} \left\| e^{\pm is\<D\>_\ell} (\partial_s c) \right\|_{W^{5+n,3/2}} \left\| \mathcal{C} \right\|_{W^{5+n,(1/3-\delta)^{-1}}}\,ds \\
& \;\;\;\;\;\;\;\;\;\;\;\;\; \lesssim \int_1^{t-1} \frac{1}{(t-s)^{3/2-3\delta}} \frac{1}{s} \frac{1}{s^{1/2+3\delta}} \|\mathcal{C}\|_X^2 \,ds \\
& \;\;\;\;\;\;\;\;\;\;\;\;\; \lesssim \|\mathcal{C}\|_X^2 \frac{1}{t^{3/2-3\delta}}
\end{split}
\end{equation*}
As for $II$, use repeatedly the Sobolev embedding theorem~\ref{sobolev}, Proposition~\ref{toucan} and the preliminary estimate~(\ref{flamingo}) to get
\begin{equation*}
\begin{split}
& \left\| e^{it\<D\>_\ell} \mathcal{F}^{-1} II \right\|_{W^{2,\infty}} \lesssim \left\| \mathcal{F}^{-1} II \right\|_{W^{4,2}} \\
& \;\;\;\;\;\;\;\;\;\;\;\;\; \lesssim \int_{t-1}^t \left\| T_{\widetilde{\chi}_\mathcal{O}(\xi) \chi_\mathcal{S}(\xi,\eta) m(\xi,\eta) \frac{1}{i\phi}} \left(e^{\pm is\<D\>_\ell} (\partial_s c),\mathcal{C}\right) \right\|_{W^{4,2}}\,ds \\
& \;\;\;\;\;\;\;\;\;\;\;\;\; \lesssim \int_{t-1}^t \left\| T_{\widetilde{\chi}_\mathcal{O}(\xi) \chi_\mathcal{S}(\xi,\eta) (\xi,\eta) m(\xi,\eta) \frac{1}{i\phi}} \left(e^{\pm is\<D\>_\ell} (\partial_s c),\mathcal{C}\right) \right\|_{W^{6,(1-\delta)^{-1}}}\,ds \\
& \;\;\;\;\;\;\;\;\;\;\;\;\; \lesssim \int_{t-1}^t  \left\| e^{\pm is\<D\>_\ell} (\partial_s c) \right\|_{W^{6+n,3/2}} \left\| \mathcal{C} \right\|_{W^{6+n,(1/3-\delta)^{-1}}}\,ds \\
& \;\;\;\;\;\;\;\;\;\;\;\;\; \lesssim \int_{t-1}^t \frac{1}{s} \frac{1}{s^{1/2+3\delta}} \|\mathcal{C}\|_X^2 \,ds \\
& \;\;\;\;\;\;\;\;\;\;\;\;\; \lesssim \|\mathcal{C}\|_X^2 \frac{1}{t^{3/2+3\delta}}
\end{split}
\end{equation*}

\subsection{The term away from $\mathcal{S}$~(\ref{eagle3})}

First transform this term by an integration by parts using the identity
$\frac{\partial_\eta \phi}{is |\partial_\eta \phi|^2}\cdot\partial_\eta e^{is\phi} = e^{is\phi}$. This gives
\begin{subequations}
\begin{align}
\label{seagull1}
(\ref{eagle3}) = & - \int_1^t  \int \widetilde{\chi}_\mathcal{O}(\xi) \chi_{\mathcal{T}}(\xi,\eta) \frac{\partial_\eta \phi}{is |\partial_\eta \phi|^2} m(\xi,\eta) e^{i s \phi} \partial_\eta \widehat{c} (\eta) \widehat{c} (\xi-\eta) \, d\eta\,ds \\ 
\label{seagull2}
& - \int_1^t \int \widetilde{\chi}_\mathcal{O}(\xi) \chi_{\mathcal{T}}(\xi,\eta) \frac{\partial_\eta \phi}{is |\partial_\eta \phi|^2} \partial_\eta m(\xi,\eta) e^{i s \phi} \widehat{c} (\eta) \widehat{c} (\xi-\eta) \, d\eta\,ds \\
& \quad \quad \quad \mbox{ + \{symmetric and easier terms\}}.
\end{align}
\end{subequations}
Let us begin with $(\ref{seagull1})$, which we split as follows:
\begin{equation}
\begin{split}
- (\ref{seagull1}) & = \int_1^t  \int \widetilde{\chi}_\mathcal{O}(\xi) \chi_{\mathcal{T}}(\xi,\eta) \frac{\partial_\eta \phi}{is |\partial_\eta \phi|^2} m(\xi,\eta) e^{i s \phi} \partial_\eta \widehat{c} (\eta) \widehat{c} (\xi-\eta) \, d\eta\,ds \\
& = \int_1^{t/2} + \int_{t/2}^t \dots \overset{def}{=} I + II.
\end{split}
\end{equation}
For $\delta<0$, $|\delta|$ small, apply successively the Sobolev embedding theorem, the dispersive estimate~(\ref{crane}), and Proposition~\ref{toucan} to get
\begin{equation}
\begin{split}
\left\| \mathcal{F}^{-1} I \right\|_{W^{2,\infty}} & \lesssim \int_1^{t/2} \left\| e^{i(t-s)\<D\>} \frac{1}{s} T_{\widetilde{\chi}_\mathcal{O}(\xi) \chi_{\mathcal{T}}(\xi,\eta) \frac{\partial_\eta \phi}{i |\partial_\eta \phi|^2} m(\xi,\eta)} \left( e^{\pm is \<D\>_\ell} (xc),\mathcal{C}\right) \right\|_{W^{3,\left( \frac{1}{6} + \delta \right)^{-1}}} \,ds \\
& \lesssim \int_1^{t/2}  \frac{1}{(t-s)^{1-3\delta}} \frac{1}{s} \left\| T_{ \widetilde{\chi}_\mathcal{O}(\xi) \chi_{\mathcal{T}}(\xi,\eta)  \frac{\partial_\eta \phi}{i |\partial_\eta \phi|^2} m(\xi,\eta)}
\left( e^{\pm is \<D\>_\ell} (xc),\mathcal{C}\right) \right\|_{W^{5,\left( \frac{5}{6} - \delta \right)^{-1}}}\,ds\\
& \lesssim \int_1^{t/2} \frac{1}{(t-s)^{1-3\delta}} \frac{1}{s} \left\|xc\right\|_{H^{n+5}} \left\|\mathcal{C}\right\|_{W^{n+5,\left(\frac{1}{3}-\delta \right)^{-1}}}\,ds\\
& \lesssim \int_1^{t/2} \frac{1}{(t-s)^{1-3\delta}} \frac{1}{s} \|\mathcal{C}\|_X^2 \sqrt{s} \frac{1}{s^{\frac{1}{2}+3\delta}} \,ds \\
& \lesssim \|\mathcal{C}\|_X^2 \frac{1}{t}.
\end{split}
\end{equation}
Taking this time $\delta > 0$ and small, and retracing the above steps, one gets
$$
\left\| \mathcal{F}^{-1} II \right\|_{W^{2,\infty}} \lesssim \|\mathcal{C}\|_X^2 \frac{1}{t}.
$$
The term~(\ref{seagull2}) can be estimated in a very similar way. Indeed, since $m$ satisfies the estimates~(\ref{wigeon}), $\partial_\eta m(\xi,\eta)$ yields at worst singularities of the type $\frac{1}{|\eta|}$, $\frac{1}{|\xi-\eta|}$. The above scheme can then be employed since by Hardy's inequality, and Plancherel's equality, $\left\| \frac{1}{|\xi|} \widehat{c}(\xi) \right\|_2 \lesssim \left\| \partial_\xi \widehat{c}(\xi) \right\|_2 = \left\| xc \right\|_2$.

\section{Localization estimates}

\label{sectionlocalization}

We want to prove here that
$$
\sup_t \left[ \frac{1}{\sqrt{t}} \left\| |x| (a,b) \right\|_{H^{N'}} + \left\| |x|^{1/8} \widetilde{Z}_{\mathcal{O}} (a,b) \right\|_2 \right] \lesssim \left\|(e^{it\<D\>_{c_s}} \mathcal{A}_0 , e^{it\<D\>} \mathcal{B}_0 ) \right\|_X + \left\|(\mathcal{A},\mathcal{B} )\right\|_X^2.
$$
As above, this reduces to proving that the generic term 
$$
g(t) = \mathcal{F}^{-1} \int_0^t e^{is\phi(\xi,\eta)} m(\eta,\xi) \widehat{c}(\eta,s) \widehat{c}(\xi-\eta,s) \, ds
$$
defined in~(\ref{pinguin}) satisfies the localization estimates
\begin{equation}
\label{tit}
\sup_t \left[ \frac{1}{\sqrt{t}} \left\| |x| g \right\|_{H^{N'}} + \left\| |x|^{1/8} \widetilde{Z}_{\mathcal{O}} g \right\|_2 \right] \lesssim \left\|(\mathcal{A},\mathcal{B} )\right\|_X^2.
\end{equation}
By symmetry, it suffices to control
$$
g'(t) = \mathcal{F}^{-1} \int_0^t e^{is\phi(\xi,\eta)} m(\eta,\xi) \zeta^1(\xi,\eta) \widehat{c}(\eta,s) \widehat{c}(\xi-\eta,s)\,ds
$$
(where the cut-off symbol $\zeta^1$, defined in Section~\ref{sectionmachinery}, ensures that $|\xi-\eta|\lesssim |\eta|$ for $(\xi,\eta)$ large). The bound for the second norm in~(\ref{tit}) was derived in~\cite{PG}, and the same scheme of proof applies here (once again, the novelty compared to~\cite{PG} is that the Euler-Maxwell system is quasilinear, but the resulting loss of regularity in the nonlinear term is easily absorbed by the $H^N$ norm). Therefore, we focus on the first norm in~(\ref{tit}), for which some new difficulties arise. It will be helpful to split $m$ as $m = m_0 m_1 m_2$ (see Subsection~\ref{duhamel}).
Multiplying $g$ by the weight $x$ corresponds in Fourier space to differentiating $\widehat{g}$ with respect to $\xi$. This gives
\begin{subequations}
\begin{align}
\label{redtailhawk1} \partial_\xi \widehat{g'}(\xi) = & \int_0^t \int e^{is\phi} m(\xi,\eta) \zeta^1(\xi,\eta) \widehat{c}(\eta) \partial_\xi \widehat{c}(\xi-\eta) \,d\eta\,ds \\
\label{redtailhawk3} & + \int_0^t \int is \partial_\xi \phi e^{is\phi} m(\xi,\eta) \zeta^1(\xi,\eta) \widehat{c}(\eta) \widehat{c}(\xi-\eta) \,d\eta\,ds \\
\label{redtailhawk2b} & + \int_0^t \int e^{is\phi} m_0(\xi) m_1(\eta) \partial_\xi m_2(\xi-\eta) \zeta^1(\xi,\eta) \widehat{c}(\eta) \widehat{c}(\xi-\eta) \,d\eta\,ds \\
\label{redtailhawk2c} & + \int_0^t \int e^{is\phi} m_0(\xi) m_1(\eta) m_2(\xi-\eta) \partial_\xi \zeta^1(\xi,\eta) \widehat{c}(\eta) \widehat{c}(\xi-\eta) \,d\eta\,ds \\
\label{redtailhawk2} & + \partial_\xi m_0(\xi) \int_0^t \int e^{is\phi} m_1(\eta) m_2(\xi-\eta) \zeta^1(\xi,\eta) \widehat{c}(\eta) \widehat{c}(\xi-\eta) \,d\eta\,ds.
\end{align}
\end{subequations}

\subsection{Estimate of (\ref{redtailhawk1})}

To estimate~(\ref{redtailhawk1}), use the Strichartz estimate~(\ref{strichartz}) and Proposition~\ref{toucan} to get
\begin{equation*}
\begin{split}
\left\| \mathcal{F}^{-1} (\ref{redtailhawk1}) \right\|_{H^{N'}} & \lesssim  \left\|\int_1^{t} e^{is\<D\>} T_{m(\xi,\eta) \zeta^1(\xi,\eta)} (\mathcal{C},e^{is\<D\>} xc)\,ds \right\|_{H^{N'}}  \\
& \lesssim \left\|T_{m(\xi,\eta) \zeta^1(\xi,\eta)} (\mathcal{C},e^{is\<D\>} xc)\right\|_{L^{\left(\frac{1}{2}+\frac{3}{2}\delta_1\right)^{-1}}_t W^{N'+1,\left(\frac{5}{6}-\delta_1\right)^{-1}}_x} \\
& \lesssim \left\| \left\| \mathcal{C} \right\|_{W^{n+N'+1,\left(\frac{1}{3} - \delta_1 \right)^{-1}}} \left\|x c\right\|_2 \right\|_{L^{\left(\frac{1}{2}+\frac{3}{2}\delta_1\right)^{-1}}_t} \\
& \lesssim \left\| \mathcal{C} \right\|_X^2 \left\| \sqrt{\<s\>} \< s \>^{-\frac{1}{2}-3\delta_1} \right\|_{L^{\left(\frac{1}{2}+\frac{3}{2}\delta_1\right)^{-1}}_t} \\
& \lesssim \left\| \mathcal{C} \right\|_X^2 \sqrt{\< t \>}.
\end{split}
\end{equation*}

\subsection{Estimate of (\ref{redtailhawk3})}

To estimate~(\ref{redtailhawk3}), distinguish three types of interactions, by writing $c = Z_\mathcal{O} c + \widetilde{Z}_\mathcal{O} c$. The term~(\ref{redtailhawk2}) can be written as
\begin{subequations}
\begin{align}
\label{hummingbird1}
\mathcal{F}^{-1} (\ref{redtailhawk2}) = & \int_0^t e^{is\<D\>} s T_{m(\xi,\eta) \zeta^1(\xi,\eta) \partial_\xi \phi} (Z_\mathcal{O} \mathcal{C},Z_\mathcal{O} \mathcal{C})\,ds \\
\label{hummingbird2}
& + \int_0^t e^{is\<D\>} s T_{m(\xi,\eta) \zeta^1(\xi,\eta) \partial_\xi \phi} (\widetilde{Z}_\mathcal{O} \mathcal{C},\widetilde{Z}_\mathcal{O} \mathcal{C})\,ds \\
\label{hummingbird3}
& +  \int_0^t e^{is\<D\>} s T_{m(\xi,\eta) \zeta^1(\xi,\eta)\partial_\xi \phi} (\widetilde{Z}_\mathcal{O} \mathcal{C},Z_\mathcal{O} \mathcal{C})\,ds\\
\label{hummingbird4}
& + \int_0^t e^{is\<D\>} s T_{m(\xi,\eta) \zeta^1(\xi,\eta)\partial_\xi \phi} (Z_\mathcal{O} \mathcal{C},\widetilde{Z}_\mathcal{O} \mathcal{C})\,ds.
\end{align}
\end{subequations}
The term~(\ref{hummingbird1}) can be treated exactly as in~\cite{PG}, thus we skip it. Next we shall bound the term~(\ref{hummingbird3}). The term~(\ref{hummingbird2}) is comparatively easier, since the two interacting waves correspond to non-outcome frequencies, thus enjoying better bounds. As for the term~(\ref{hummingbird4}) it is also easier: indeed for this term, the symbol $\zeta^1(\xi,\eta)$ imposes that $\widetilde{Z}_\mathcal{O} \mathcal{C}$ is lower frequency than $Z_\mathcal{O} \mathcal{C}$; but this is possible only if both are low frequency.

Coming back to~(\ref{hummingbird3}), use Proposition~\ref{toucan} to get
\begin{equation*}
\begin{split}
\left\| \ref{hummingbird3} \right\|_{H^{N'}} & \lesssim \int_0^t s \left\| e^{is\<D\>} T_{m(\xi,\eta) \zeta^1(\xi,\eta) \partial_\xi \phi} (\widetilde{Z}_\mathcal{O} \mathcal{C}, Z_\mathcal{O} \mathcal{C})\right\|_{H^{N'}}\,ds \\
& \lesssim \int_0^t s \left\| Z_\mathcal{O} \mathcal{C} \right\|_{L^{\left( \frac{1}{3}-\delta_1 \right)^{-1}}} \left\| \widetilde{Z}_\mathcal{O} \mathcal{C} \right\|_{W^{N'+n,\left(\frac{1}{6}+\delta_1\right)^{-1}}} \,ds \\
& \lesssim \|\mathcal{C}\|_X^2 \int_0^t s \frac{1}{s^{\frac{1}{2}+3\delta_1}} \frac{1}{s^{1-3\delta_1}}\,ds \lesssim \|\mathcal{C}\|_X^2 \sqrt{t}.
\end{split}
\end{equation*}

\subsection{Estimate of (\ref{redtailhawk2b})} 

By~(\ref{wigeon}), $\partial_\xi m_2(\xi-\eta)$ can be bounded by $\frac{C}{|\xi-\eta|}$. Bounding by Hardy's inequality $\frac{1}{|\xi-\eta|} \widehat{c}(\xi-\eta)$ in $L^2$ by $\partial_\xi \widehat{c}(\xi-\eta)$ in $L^2$, the estimate for (\ref{redtailhawk1}) can be easily adapted.

\subsection{Estimate of (\ref{redtailhawk2c})}

Since $\partial_\xi \zeta^1(\xi,\eta)$ does not have a singularity, this term is easy and we skip it.

\subsection{Estimate of (\ref{redtailhawk2})}

By~(\ref{wigeon}), $\partial_\xi m_0(\xi)$ can be bounded by $1$ for high frequencies, and $\frac{1}{|\xi|}$ for small frequencies. Forgetting about high frequencies, which are easily dealt with, we need to bound
$$
\mathcal{F}^{-1} \frac{1}{|\xi|}  \int_0^t \int e^{is\phi} \partial_\xi m_1(\eta) m_2(\xi-\eta) \zeta^1(\xi,\eta) \widehat{c}(\eta) \widehat{c}(\xi-\eta) \,d\eta\,ds
$$
in $H^{N'}$. By Hardy's inequality, it suffices to bound
$$
\mathcal{F}^{-1} \partial_\xi \int_0^t \int e^{is\phi} \partial_\xi m_1(\eta) m_2(\xi-\eta) \zeta^1(\xi,\eta) \widehat{c}(\eta) \widehat{c}(\xi-\eta) \,d\eta\,ds
$$
in $H^{N'}$. But expanding the $\xi$ derivative yields terms similar to~(\ref{redtailhawk1}) (\ref{redtailhawk3}) (\ref{redtailhawk2b}) (\ref{redtailhawk2c}), which we have just seen how to estimate.

\section{Energy estimates for the Maxwell  part}

\label{plasma}

We shall prove in this section that
\begin{equation}
\label{petrel1}
\left\| \mathcal{B} \right\|_{H^N} \lesssim \left\|(\mathcal{A}_0,\mathcal{B}_0)\right\|_{H^N} + \left\|(\mathcal{A},\mathcal{B})\right\|_X^2 + \int_0^t \frac{1}{\<s\>} \left\| (\mathcal{A},\mathcal{B}) \right\|_{H^N} \,ds.
\end{equation}
Together with~(\ref{petrel2}), this will imply that
$$
\left\| ( \mathcal{A},\mathcal{B} ) \right\|_{H^N} \lesssim \epsilon 
t^{C_0 \epsilon}.
$$
The following observation will be crucial: it follows from their definition that $\mathcal{A}$ and $\mathcal{B}$ control the physical unknowns $u$ and $n$ as follows:
\begin{equation}
\begin{split}
& \left\|Qu \right\|_{H^N} \lesssim \left\|\mathcal{A} \right\|_{H^N} \\
& \left\|Pu\right\|_{H^{N+1}} \lesssim \left\| \mathcal{B} \right\|_{H^N} \\
& \left\|n\right\|_{H^{N}} \lesssim \left\| \mathcal{A} \right\|_{H^N}.
\end{split}
\end{equation}

\subsection{Preliminary estimate: $\partial_s a$}

It follows from $(EM')$ that 
$$
\partial_t a = e^{-it\<D\>_{c_s}} \left[ \frac{\langle D \rangle \nabla}{|D|} \cdot (n u) + i |D| \left( \left| u \right|^2 + |n|^2 \right)\right]
$$
Therefore by the product estimates~(\ref{product}),
\begin{equation}
\label{snake}
\begin{split}
\left\|\partial_t a \right\|_{H^{N-1}} & \lesssim \left\|nu\right\|_{H^{N}} + \left\| u^2 \right\|_{H^N} + \left\| n^2 \right\|_{H^N} \\
& \lesssim \left( \left\| n \right\|_{H^N} + \left\| u \right\|_{H^N} \right) \left( \left\| n \right\|_{\infty} + \left\| u \right\|_{\infty} \right) \\
& \lesssim \frac{1}{\< t \>^{\frac{1}{2}+3\delta_1}}
\end{split}
\end{equation}

\subsection{Distinction between outcome and non-outcome frequencies}

Consider the integral equation satisfied by $b$:
$$
b(t) = \mathcal{B}_0 + \int_0^t e^{- is\langle D \rangle} \frac{\nabla}{\<D\>} \times ( n u)\,ds.
$$
Split $n$ and $u$ into $Z_{\mathcal{O}} n + \widetilde{Z}_{\mathcal{O}} n$, respectively $Z_{\mathcal{O}} u + \widetilde{Z}_{\mathcal{O}} u$. This gives
\begin{subequations}
\begin{align}
\label{rabbit1} b(t) = \mathcal{B}_0 & + \int_0^t e^{- is\langle D \rangle} \frac{\nabla}{\<D\>} \times (\widetilde{Z}_{\mathcal{O}} n \,\widetilde{Z}_{\mathcal{O}} u)\,ds \\
\label{rabbit2} & + \int_0^t e^{- is\langle D \rangle} \frac{\nabla}{\<D\>} \times (Z_{\mathcal{O}} n \,Z_{\mathcal{O}} u)\,ds \\
\label{rabbit3} & + \int_0^t e^{- is\langle D \rangle} \frac{\nabla}{\<D\>} \times (Z_{\mathcal{O}} n \, \widetilde{Z}_{\mathcal{O}} u + \widetilde{Z}_{\mathcal{O}} n Z_{\mathcal{O}} u)
\end{align}
\end{subequations}
The term~(\ref{rabbit1}) is easily estimated: using standard product laws, and the Sobolev embedding theorem:
\begin{equation*}
\begin{split}
\left\| (\ref{rabbit1}) \right\|_{H^N} & = \left\| \int_0^t e^{- is\langle D \rangle} \frac{\nabla}{\<D\>} \times (\widetilde{Z}_{\mathcal{O}} n \,\widetilde{Z}_{\mathcal{O}} u)\,ds \right\|_{H^N} \\
& \lesssim \int_0^t \left[ \left\| \widetilde{Z}_{\mathcal{O}} n \right\|_{H^N} \left\| \widetilde{Z}_{\mathcal{O}} u \right\|_{L^\infty} + \left\| \widetilde{Z}_{\mathcal{O}} n \right\|_{L^\infty} \left\| \widetilde{Z}_{\mathcal{O}} u \right\|_{H^N} \right] \,ds \\
& \lesssim \left\| (\mathcal{A},\mathcal{B}) \right\|_X \int_0^t \frac{1}{\<s\>} \left\|(u,n)\right\|_{H^N} \,ds\\
& \lesssim \left\| (\mathcal{A},\mathcal{B}) \right\|_X \int_0^t \frac{1}{\<s\>} \left\|(\mathcal{A},\mathcal{B})\right\|_{H^N} \,ds\\
\end{split}
\end{equation*}
For the term (\ref{rabbit2}), we take advantage of the frequency localization of $Z_{\mathcal{O}} n$ and $Z_{\mathcal{O}} u$ to write, with the help of Bernstein's inequality,
\begin{equation*}
\begin{split}
\left\| (\ref{rabbit2}) \right\|_{H^N} & = \left\| \int_0^t e^{- is\langle D \rangle} \frac{\nabla}{\<D\>} \times (Z_{\mathcal{O}} n \,Z_{\mathcal{O}} u)\,ds \right\|_{H^N} \\
& \lesssim \int_0^t \left\| Z_{\mathcal{O}} n \,Z_{\mathcal{O}} u \right\|_{\left(\frac{2}{3}-2\delta\right)^{-1}}\,ds \\
& \lesssim \int_0^t \left\| Z_{\mathcal{O}} n \right\|_{\left(\frac{1}{3}-\delta\right)^{-1}} \left\| Z_{\mathcal{O}} u \right\|_{\left(\frac{1}{3}-\delta\right)^{-1}} \,ds \\
& \lesssim \|(\mathcal{A},\mathcal{B})\|_X^2 \int_0^t \frac{ds}{\<s\>^{1+6\delta_1}}\,ds \lesssim \|(\mathcal{A},\mathcal{B})\|_X^2.
\end{split}
\end{equation*}

\subsection{Interactions between outcome and non-outcome frequencies}

Thus we now take a closer look at~(\ref{rabbit3}), wich reads
\begin{subequations}
\begin{align}
\label{tiger1} (\ref{rabbit3}) = & \int_0^t e^{- is\langle D \rangle} \frac{\nabla}{\<D\>} \times ( Z_{\mathcal{O}} n \, \widetilde{Z}_{\mathcal{O}} P u)\,ds \\
\label{tiger2} & + \int_0^t e^{- is\langle D \rangle} \frac{\nabla}{\<D\>} \times (Z_{\mathcal{O}} n \, \widetilde{Z}_{\mathcal{O}} Q u)\,ds \\
\label{tiger3} & + \int_0^t e^{- is\langle D \rangle} \frac{\nabla}{\<D\>} \times (\widetilde{Z}_{\mathcal{O}} n \, Z_{\mathcal{O}} u )\,ds 
\end{align}
\end{subequations}
The first term, (\ref{tiger1}), can be estimated with the help of the Strichartz estimate~(\ref{strichartz}) and the standard product law~\ref{product}:
\begin{equation*}
\begin{split}
\left\| (\ref{tiger1}) \right\|_{H^N} & = \left\| \int_0^t e^{- is\langle D \rangle} \frac{\nabla}{\<D\>} \times ( Z_{\mathcal{O}} n \, \widetilde{Z}_{\mathcal{O}} P u)\,ds \right\|_{H^N} \\
&\lesssim \left\| Z_{\mathcal{O}} n \, \widetilde{Z}_{\mathcal{O}} P u \right\|_{L_t^{\left( \frac{1}{2}+\frac{3}{2}\delta_1 \right)^{-1}} W_x^{\left( 1 - \frac{5}{2}\delta_1 + N \right) , \left( \frac{5}{6} - \delta_1 \right)^{-1}}} \\
&\lesssim \left\| \left\| Z_{\mathcal{O}} n \right\|_{L_x^{\left( \frac{1}{3}-\delta_1 \right)^{-1}}} \left\| \widetilde{Z}_{\mathcal{O}} P u \right\|_{H_x^{N+1}} + \left\| Z_{\mathcal{O}} n \right\|_{H^{N+1}_x} \left\| \widetilde{Z}_{\mathcal{O}} P u \right\|_{L_x^{\left( \frac{1}{3}-\delta_1 \right)^{-1}}} \right\|_{L_t^{\left( \frac{1}{2}+\frac{3}{2}\delta_1 \right)^{-1}}} \\
&\lesssim \left\|(\mathcal{A},\mathcal{B})\right\|_X^2 \left\| \<t\>^{-\frac{1}{2}-3\delta_1} \right\|_{L_t^{\left( \frac{1}{2}+\frac{3}{2}\delta_1 \right)^{-1}}} \\
& \lesssim \left\|(\mathcal{A},\mathcal{B})\right\|_X^2.
\end{split}
\end{equation*}
The estimates for the terms (\ref{tiger2}) and  (\ref{tiger3}) are exactly the same, 
changing the roles of $n$ and $Qu$. We will only treat (\ref{tiger2}). 
The term  (\ref{tiger2})  can be decomposed as
\begin{subequations}
\begin{align}
\label{dragon1} (\ref{tiger2}) & = \int_0^t e^{- is\langle D \rangle} \frac{\nabla}{\<D\>} \times (Z_{\mathcal{O}} n \, \widetilde{Z}_{\mathcal{O}} Z_l Q u)\,ds \\
\label{dragon2} & + \int_0^t e^{- is\langle D \rangle} \frac{\nabla}{\<D\>} \times (Z_{\mathcal{O}} n \, Z_h Q u)\,ds
\end{align}
\end{subequations}
(recall that $Z_l$ and $Z_h$ have been defined in Section~\ref{sectionmachinery}, and that $M_0$ has been chosen so big that $Z_h \widetilde{Z}_{\mathcal{O}} = Z_h$).
The term (\ref{dragon1}) can be estimated exactly as (\ref{rabbit2}); thus we skip it. We are left with (\ref{dragon2}). Recall now that $n$ and $Qu$ can be written as
$$
n(t) = 2 \frac{|D|}{\<D\>_{c_s}} \frak{Re} e^{it\<D\>_{c_s}} a(t) \;\;\;\;\mbox{and}\;\;\;\; Qu(t) = - 2 \frac{\nabla}{|D|} \frak{Im} e^{it\<D\>_{c_s}} a(t).
$$
This implies that the Fourier transform of (\ref{dragon2}) can be written as a sum of terms of the type 
\begin{equation}
\label{ox}
\mathcal{F} (\ref{dragon2}) = \int_0^t \int e^{is \phi(\xi,\eta)} \chi_\mathcal{O}(\eta) \left(1-\chi \left( \frac{\xi-\eta}{M} \right) \right) \widetilde{m}(\xi,\eta) \widehat{a}(\eta,s) \widehat{a}(\xi-\eta,s)\,d\eta\,ds
\end{equation}
where, for simplicity, we denote undistinctly $\widehat{a}$ for $\widehat{a}$ and $\widehat{\bar a}$, $\phi$ has the form
$$
\phi(\xi,\eta) = \< \xi \> + \epsilon_1 \< \eta \>_{c_s} + \epsilon_2 \< \xi-\eta \>_{c_s} \;\;\;\;\;\mbox{with $\epsilon_1,\epsilon_2 \in \{ \pm 1 \}$},
$$
and $\widetilde{m}(\xi,\eta)$ is a (matrix-valued) symbol satisfying the estimates
$$
\left| \partial_\xi^\alpha \partial_\eta^\beta \widetilde{m}(\eta,\xi) \right| \lesssim \frac{1}{|(\xi,\eta)|^{|\alpha|+|\beta|}}.
$$
A crucial point will be that, on the support of $\chi_\mathcal{O}(\eta) \left(1-\theta \left( \frac{\xi-\eta}{M_0} \right) \right)$, since $M_0$ is chosen big enough, $|\xi|>>|\eta|\sim1$, and $\phi$ satisfies the inequalities
$$
\left| \partial_\xi^\alpha \partial_\eta^\beta \frac{1}{\phi(\eta,\xi)} \right| \lesssim \frac{1}{|\xi|^{|\alpha|+|\beta|+1}}.
$$
Integrating by parts in~(\ref{ox}) using the identity $\frac{1}{i\phi} \partial_s e^{is\phi} = e^{is\phi}$, and denoting 
$$
\mu(\xi,\eta) \overset{def}{=} \frac{\chi_\mathcal{O}(\eta) \left(1-\theta \left( \frac{\xi-\eta}{M_0} \right) \right)\widetilde{m}(\eta,\xi)}{\phi(\xi,\eta)}
$$
gives
\begin{subequations}
\begin{align}
\label{horse1} \mathcal{F} (\ref{dragon2}) = & - \int_0^t \int e^{is \phi(\xi,\eta)} \mu(\xi,\eta) \partial_s \widehat{a}(\eta,s) \widehat{a}(\xi-\eta,s)\,d\eta\,ds \\
\label{horse2} & - \int_0^t \int e^{is \phi(\xi,\eta)} \mu(\xi,\eta) \widehat{a}(\eta,s) \partial_s \widehat{a}(\xi-\eta,s)\,d\eta\,ds \\
\label{horse3} & + \int e^{it \phi(\xi,\eta)} \mu(\xi,\eta) \widehat{a}(\eta,t) \widehat{a}(\xi-\eta,t)\,d\eta\,ds \\
\label{horse4} & - \int \mu(\xi,\eta) \widehat{a}(\eta,0) \widehat{a}(\xi-\eta,0)\,d\eta\,ds.
\end{align}
\end{subequations}
The only difficult term is~(\ref{horse2}); thus we skip the other ones and estimate it with the help of Proposition~\ref{toucan}
\begin{equation*}
\begin{split}
\left\| \mathcal{F}^{-1} (\ref{horse2}) \right\|_{H^N} & = \left\| \int_0^t e^{is\<D\>} T_\mu (\mathcal{A}, e^{it\<D\>} \partial_s a(s))\,ds \right\|_{H^N} \\
& \lesssim \int_0^t \left\| T_\mu \left( \mathcal{A} \,,\,e^{it\<D\>} \partial_s a(s) \right) \right\|_{H^N} \,ds \\
& \lesssim \int_0^t \left\| \mathcal{A} \right\|_\infty \left\| \partial_s a(s) \right\|_{H^{N-1}}\,ds \\
& \lesssim \left\|(\mathcal{A},\mathcal{B})\right\|_X^2 \int_0^t \frac{1}{\<s\>} \frac{1}{\<s\>^{1/2+3\delta_1}}\,ds \lesssim \left\|(\mathcal{A},\mathcal{B})\right\|_X^2.
\end{split}
\end{equation*}

\section{Energy estimates for the acoustic part}
\label{acoustic}
We shall prove in this section that
\begin{equation}
\label{petrel2}
\left\| \mathcal{A} \right\|_{H^N} \lesssim \left\|(\mathcal{A}_0,\mathcal{B}_0)\right\|_{H^N} + \left\|(\mathcal{A},\mathcal{B})\right\|_X^2 + \int_0^t \frac{1}{\<s\>} \left\| (\mathcal{A},\mathcal{B}) \right\|_{H^N} \,ds.
\end{equation}
Together with~(\ref{petrel1}), this will imply that
$$
\left\| ( \mathcal{A},\mathcal{B}) \right\|_{H^N} \lesssim \epsilon  t^{C_0 \epsilon}.
$$

\subsection{The equation  \eqref{Aequat}}
 First we rewrite the 
evolution equation \eqref{Aequat} satisfied by $\A$. 
We will use the notation $\d =  \frac{\langle D \rangle_{c_s} \nabla}{|D|}  $. 
We start by expanding the first nonlinear terms appearing in \eqref{Aequat} 

We start by taking $N$ derivatives of  \eqref{Aequat}. We get 

\begin{equation}\label{NAequat}
\begin{aligned}
   \partial_t  \partial^N \mathcal{A} &=   i \langle D \rangle_{c_s} \partial^N  \mathcal{A} - 
\frac{1}{2}\d \cdot ( \partial^N   n u   + n \partial^N  u )  \\ 
 & \quad       +  \frac{i |D|}2  \left( u  \partial^N  u  + c_s^2 n  \partial^N  n  \right) + R^N_1  \\ 
&=    i \langle D \rangle_{c_s} \partial^N  \mathcal{A} - 
\frac{1}{2} u \cdot \d \partial^N   n 
  - \frac{1}{2} n \d \cdot     \partial^N  u   \\ 
 & \quad       +  i u \cdot   \frac{|D|}2  \partial^N  u  +  i c_s^2 n  \frac{|D|}2    \partial^N  n   + R^N_2 \\ 
& =    i \langle D \rangle_{c_s} \partial^N  \mathcal{A} - u\cdot   \nabla \partial^N \mathcal{A}  
 + i n \langle D \rangle_{c_s} \partial^N  \mathcal{A}  + R^N_3 
\end{aligned}
\end{equation} 
where  the rest terms $R^N_i $ consist  of quadratic  lower order terms.  In particular, we have

\begin{align*}
   R^N_1  &=  \d \cdot \Big(   \partial^N  ( n u) -       \partial^N   n u   + n \partial^N  u \Big)  \\ 
  & \quad    +  \frac{i |D|}4  \left( \partial^N (|u|^2)   - 2    u  \partial^N  u  + 
  c_s^2   \partial^N (|n|^2)     - 2 c_s^2 n  \partial^N  n  \right)   \\ 
2 R^N_2     &  = 2R^N_1  - [ \d, u   ]  \partial^N   n  -   [ \d\cdot  , n   ] \partial^N   u   \\ 
 & \quad   + i [|D|,u ]  \partial^N   u   + i c_s^2[|D|,n ]  \partial^N   n \\ 
   R^N_3  &  =  R^N_2   - i  n  \frac1{2|D|}    \partial^N  n  + i u   \frac{|D|}2 \partial^N  Pu := R^N 
\end{align*}
It is clear that we have the following estimate 
$$  \| R^N   \|_{L^2}  \leq \| \nabla (\A,\B) \|_{L^\infty}   \|  (\A,\B) \|_{H^N} .      $$
Hence $ R^N   $  does not loose derivatives, but for the part of $(\A,\B)  $ which is out-come 
we do not have enough decay to use Gronwall directly.

\subsection{Non resonant phase}

Due to the slow decay of the $Z_{\mathcal{O}} u   $ and $Z_{\mathcal{O}} n    $, we have to use 
non resonant properties of the  second and third terms on the right-hand side of \eqref{Aequat}.

\begin{lem}\label{lemNon}
 There exist a positive number $\kappa_0 > 0$ and 
 a constant $C_0 > 0$ such that for 
$|\xi| \geq C_0 $ and   $|\eta| \leq C_\mathcal{R} $,  we have 
\begin{equation} \label{phase2}
\left| \partial_\xi^\alpha \partial_\eta^\beta \frac{1}{\phi^{\epsilon_1,\epsilon_2}_{c_s,k,\ell} (\xi,\eta)} \right| \lesssim \frac{1}{|\xi|^{|\alpha|}}
\end{equation}
for $\epsilon_1,\epsilon_2 = \pm$ and $k,\ell = 1,c_s$. 
\end{lem}
\begin{proof}
We will only consider the phase  $\phi^{ + , -}_{c_s,c_s,c_s} {=} 
 \langle \xi \rangle_{c_s} +   \langle \eta \rangle_{c_s}  -  \langle \xi - \eta \rangle_{c_s}$ since 
the other phases are easier. Furthermore, we only prove the estimate on $\frac{1}{\phi}$, not its derivatives. We have 
\begin{eqnarray*}
 \langle \xi - \eta \rangle_{c_s}  &=&  c|\xi| \sqrt{ 1 + \frac{1}{c_s^2|\xi|^2}  -2 \frac{\xi.\eta}{|\xi|^2} 
   + \frac{|\eta|^2}{|\xi|^2}   }  \\ 
 &=&  c_s|\xi| \left[  1 + \frac{1}{2 c_s^2|\xi|^2}  -2 \frac{\xi.\eta}{2 |\xi|^2} 
   + \frac{|\eta|^2}{2 |\xi|^2}  + O(\frac1{|\xi|^2} )    \right]. 
 \end{eqnarray*}
Hence, we see  that $  \langle \xi - \eta \rangle_{c_s}  -  c_s|\xi|  -   c_s|\eta| \leq \frac{C}{|\xi|}  $, from 
which we deduce that 
\begin{eqnarray*}
\phi^{ + , -}_{c_s,c_s,c_s} \geq   \langle \xi \rangle_{c_s} +   \langle \eta \rangle_{c_s}  -  c_s|\xi|  -   c_s|\eta| 
 -\frac{C}{|\xi|}. 
 \end{eqnarray*} 
Hence,  if $C_0$ is big enough then,
$ \phi^{ + , -}_{c_s,c_s,c_s} \geq  \frac{  \sqrt{1 +  (c_s  C_\mathcal{R} )^2}  -  c_s  C_\mathcal{R}  }{2} > 0$.
\end{proof}

\subsection{Energy estimates}

The Sobolev estimates for the Maxwell part were performed using simply Strichartz estimates and integration by parts in time depending on the cases. Due to the further loss of a derivative, this method does not apply here. Instead we will perform an iterated  energy estimate that we find interesting in its own right. 

Using that $  u  $  and $  n   $ are both real, we deduce that 

\begin{equation} \label{Nest0}
\partial_t \frac{\|  \partial^N \mathcal{A}  \|_{L^2}^2}2  = \frak{Re} \int \nabla .   u  |  \partial^N \mathcal{A}  |^2 
 + i [ \langle D \rangle_{c_s} ,    n  ] \partial^N \overline{\A}  \partial^N \mathcal{A} +  
 R^N  \partial^N \overline{\A} .  
\end{equation}
Hence
\begin{equation}
\begin{aligned}
 \label{Nest}
& \frac{\|  \partial^N \mathcal{A}(t)   \|_{L^2}^2}2  - \frac{\|  \partial^N \mathcal{A}_0   \|_{L^2}^2}2  
 =   \\
& \quad \quad \quad   \int_0^t  \frak{Re} \int \left( \nabla .   u  |  \partial^N \mathcal{A}  |^2 
 + i [ \langle D \rangle_{c_s} ,    n  ] \partial^N \overline{\A}  \partial^N \mathcal{A} +  
 R^N  \partial^N \overline{\A} \right)\,ds.
\end{aligned}
\end{equation}
We would like now to explain how to control the three terms on the right-hand side of \eqref{Nest}.  
For the first term, we split $u$ into the outcome and non-outcome parts 
$ u =  Z_{\mathcal{O}} u   +  \widetilde Z_{\mathcal{O}} u    $. The  non-outcome part has enough 
decay to apply directly the Gronwall argument. Hence, we will only concentrate on the outcome part.   
We recall that  the profile $a(t)$ associated to $\A$ is defined by  $\A(t)  = e^{i \langle D \rangle_{c_s} t   } a(t) $. 
Also, we have
\begin{equation}\label{Zu}
  Z_{\mathcal{O}} u  =  Z_{\mathcal{O}} \frac{\nabla}{|D|\<D\>} \times  ( \frac{e^{it\<D\>} b -e^{-it\<D\>} \overline{b}   }{2i}   ) 
    + i   Z_{\mathcal{O}} \frac{\nabla}{|D|}  ( e^{i \langle D \rangle_{c_s} t   } a  - e^{i \langle D \rangle_{c_s} t   } \overline{a}   )  .    \end{equation} 
We denote by $ e^{\pm it\<D\>_l} c(t) = \C(t)  $ the divergence of 
any one  of the four terms appearing in \eqref{Zu}. To control  the first term in the right-hand side of 
\eqref{Nest}, it is enough to rewrite it  in Fourier space. Hence, it is enough  to consider
\begin{equation}\label{model} 
\int_0^t  \int\! \int   e^{is\phi(\xi,\eta)} \widetilde{m}(\xi,\eta) \widehat c (s,\eta)  \widehat{\partial^N a} (s,\xi -\eta) 
 \overline{ {\widehat{\partial^N a}  }  } (s,\xi)  d\eta\, d\xi \,  ds 
\end{equation} 
where the phase $\phi$ is given by $\phi(\xi,\eta) =  \langle \xi - \eta \rangle_{c_s}  -  \langle \xi \rangle_{c_s} \pm \langle \eta \rangle_{l}$and $\widetilde{m}(\xi,\eta) \overset{def}{=} \chi_\mathcal{O} (\eta)$. Split $\widetilde{m}(\xi,\eta) = \theta\left(\frac{\xi}{M_0}\right) \widetilde{m}(\xi,\eta) + \left[ 1 - \theta\left(\frac{\xi}{M_0}\right) \right] \widetilde{m}(\xi,\eta)$. The first term corresponds to low frequencies of $\partial^N \mathcal{A}$, which are easily estimated; thus, we shall consider in the following that
$$
\widetilde{m}(\xi,\eta) =  \left[ 1 - \theta\left(\frac{\xi}{M_0}\right) \right] \chi_\mathcal{O} (\eta)
$$
From Lemma \ref{lemNon}, we know that $\phi$ is always bounded away from zero 
in the support of $\widetilde{m}$. Hence, we can integrate by parts in time (using the identity $\frac{1}{i\phi} \partial_s e^{is\phi} = e^{is\phi}$) in \eqref{model} and get 
\begin{subequations}
\begin{align}
i(\ref{Zu}) = & \label{eng1} - \int_0^t  \int\! \int  \widetilde{m}(\xi,\eta) \frac{e^{is\phi}}{\phi}  \widehat c (s,\eta)
 \partial_s    \Big( \widehat{\partial^N a} (s,\xi -\eta) 
 \overline{ {\widehat{\partial^N a}  }  } (s,\xi)        \Big)  d\eta\, d\xi\,   ds  \\ 
& \ \label{eng2} - \int_0^t  \int\! \int  \widetilde{m}(\xi,\eta) \frac{e^{is\phi}}{\phi}  \Big(  \partial_s \widehat c (s,\eta)  \widehat{\partial^N a} (s,\xi -\eta) 
 \overline{ {\widehat{\partial^N a}  }  } (s,\xi) \Big)  d\eta\, d\xi\,   ds  \\ 
& \quad  + \int\! \int    \frac{e^{it\phi}}{\phi} \widetilde{m}(\xi,\eta)  \widehat c (t,\eta)  \widehat{\partial^N a} (t,\xi -\eta) 
 \overline{ {\widehat{\partial^N a}  }  } (t,\xi)    d\eta \,d\xi  \\
& \quad \quad   -   \int\! \int    \frac{1}{\phi} \widetilde{m}(\xi,\eta) \widehat c (0,\eta)  \widehat{\partial^N a} (0,\xi -\eta) 
 \overline{ {\widehat{\partial^N a}  }  } (0,\xi)  d\eta\, d\xi 
\end{align}
\end{subequations}
We rewrite the time derivative in \eqref{eng1}  as
\begin{align}
   \Big(  \partial_s  \widehat{\partial^N a} (\xi -\eta) 
 \overline{ {\widehat{\partial^N a}  }  } (\xi)   +  \widehat{\partial^N a} (\xi -\eta) \partial_s   
 \overline{ {\widehat{\partial^N a}  }  } (\xi)        \Big)   
\end{align}
From \eqref{NAequat}, we deduce that 
\begin{equation}
e^{it\langle D \rangle_{c_s}  } \partial_t  \partial^N a = - u\cdot   \nabla \partial^N \mathcal{A}  
 + i n \langle D \rangle_{c_s} \partial^N  \mathcal{A}  + R^N.  
\end{equation} 
 Hence, \eqref{eng1}  can be expanded as 
\begin{subequations}
\begin{align}
& (\ref{eng1}) = - \int_0^t  \int\! \int  \widetilde{m}(\xi,\eta) \frac{1}{\phi}  \widehat \C(\eta)  
\left[ \widehat{u\cdot\nabla\partial^N \mathcal{A}    }  (\xi-\eta) \overline{ {\widehat{\partial^N \A}  }  } (\xi)   
+ \widehat{\partial^N \mathcal{A}    }  (\xi-\eta) \overline{ {\widehat{u\cdot\nabla\partial^N \A}  }  } (\xi)  
\right]  \label{cous1} d\eta\, d\xi\,   ds  \\ 
& \quad  - \int_0^t  \int\! \int  \widetilde{m}(\xi,\eta) \frac{i}{\phi}  \widehat \C(\eta)  
\left[ \widehat{ n \langle D \rangle_{c_s}  \partial^N \mathcal{A}    }  (\xi-\eta) \overline{ {\widehat{\partial^N \A}  }  } (\xi)   
-  \widehat{\partial^N \mathcal{A}    }  (\xi-\eta) \overline{ {\widehat{  n \langle D \rangle_{c_s}  \partial^N \A}  }  } (\xi) 
\right]  \label{cous2} d\eta\, d\xi\,   ds    \\ 
& \quad - \int_0^t  \int\! \int \widetilde{m}(\xi,\eta)  \frac{1}{\phi}  \widehat \C(\eta) 
\left[ \widehat{R^N  }  (\xi-\eta) \overline{ {\widehat{\partial^N \A}  }  } (\xi)   
-  \widehat{\partial^N \mathcal{A}    }  (\xi-\eta) \overline{ {\widehat{ R^N}   } }(\xi) 
\right]d\eta\, d\xi\,   ds 
\end{align}
\end{subequations}
The difficulty in bounding \eqref{cous1} is that $\A $ appears with $N+1$ derivatives. The main 
idea is to use some sort of energy estimate to perform an integration by part so that the 
extra derivative can be moved on a term with fewer derivatives.  Keeping in mind that $u$ is real-valued, we can rewrite \eqref{cous1}  as 
\begin{align*}
\eqref{cous1} = & - \int_0^t  \int\!  \int\! \int  \widetilde{m}(\xi,\eta) \frac{i}{\phi}  \widehat \C(\eta)  
\left[ \widehat {u}(\nu) \cdot ( \xi-\eta -\nu )  \widehat{\partial^N \mathcal{A}    }  (\xi-\eta -\nu  ) 
 \widehat { \partial^N  { \overline {\A}  }  } ( - \xi)   \right. \\
& \quad \quad  \left.  + \widehat{\partial^N \mathcal{A}    }  (\xi-\eta)  \widehat{u}(\nu) 
 \cdot (-\xi-\nu)  \widehat { \partial^N  {\overline {\A}  }  } (-\xi-\nu)  \right]  d\eta \, d\xi \, d\nu \, ds \\ 
= & -i \int_0^t  \int\!  \int\! \int    \widehat \C(\eta)  \widehat {u}(\nu) \cdot
\left[ \mu(\xi,\eta) ( \xi-\eta -\nu ) \widehat{\partial^N  \mathcal{A}    }  (\xi-\eta -\nu  ) 
 \widehat { \partial^N  { \overline {\A}  }  } ( - \xi)   \right. \\
& \quad \quad \left. + \mu(\xi-\nu,\eta)   \widehat{\partial^N \mathcal{A}    }  (\xi-\eta - \nu) 
 (-\xi)  \widehat { \partial^N  {\overline {\A}  }  } (-\xi)  \right]  d\eta \, d\xi  \,d\nu \, ds 
 \end{align*}
 where $\mu(\xi,\eta) = \frac{\widetilde{m}(\xi,\eta)}{\phi (\xi,\eta) }$  and we made the change of variable 
$\xi \to \xi -\nu$ in the last line. The integrand of the term appearing in the last two lines can be 
rewritten as 
\begin{align*}
 \widehat \C(\eta)  \widehat {u}(\nu)  \widehat{\partial^N \mathcal{A}    }  (\xi-\eta - \nu) 
 \widehat { \partial^N  {\overline {\A}  }  } (-\xi) \cdot  
\left[   \mu(\xi,\eta)  ( \xi-\eta -\nu )  +  \mu(\xi-\nu,\eta)  \xi    \right]
\end{align*}
and the term between brackets is equal to 
\begin{equation}
\label{egret}
M(\xi,\eta,\nu) = -  \mu(\xi,\eta)   (\eta+ \nu )   +   \left[   \mu(\xi,\eta)  -   \mu(\xi-\nu,\eta)     \right] \xi .
\end{equation}
Proposition~\ref{coot} gives the desired conclusion, namely that
\begin{equation}
\begin{split}
\left\|(10.10a)\right\|_2 & \lesssim \int_0^t \left| \int \int \int M(\xi,\eta,\nu) \widehat \C(\eta)  \widehat {u}(\nu) \widehat{\partial^N  \mathcal{A} }  (\xi-\eta -\nu  ) 
 \widehat { \partial^N  { \overline {\A}  }  }(-\xi)  d\eta \, d\xi\, d\nu  \right| \,ds \\
& \lesssim \int_0^t \|u\|_\infty \|\mathcal{C}\|_\infty \left\|\partial^N  \mathcal{A}\right\|_2^2 \,ds \\
& \lesssim \left\|(\mathcal{A},\mathcal{B}) \right\|_X^4 \int_0^t \frac{1}{\<s\>^{1+6 \delta_1}} \,ds \lesssim \left\|(\mathcal{A},\mathcal{B}) \right\|_X^4
\end{split}
\end{equation}
The treatement of \eqref{cous2} is very similiar and we do not detail it here.   \vspace{.2cm}

 To control  the second  term on the right-hand side of  
\eqref{Nest}, we  rewrite it  in Fourier space. Hence, it is enough  to consider 
\begin{equation}\label{model2} 
\int_0^t  \int\! \int   e^{is\phi}  \widehat c (\eta)   (  \langle \xi - \eta \rangle_{c_s}  -  \langle \xi \rangle_{c_s}     ) 
  \widehat{\partial^N a} (\xi -\eta) 
 \overline{ {\widehat{\partial^N a}  }  } (\xi)  d\eta d\xi   ds 
\end{equation} 
where the phase $\phi$ is given by $\phi(\xi,\eta) =  \langle \xi - \eta \rangle_{c_s}  -  \langle \xi \rangle_{c_s}  
 \pm \langle \eta \rangle_{c_s}    $ and $ e^{\pm i \langle D \rangle_{c_s} t  } c  $ is one of the two terms 
appearing in the decomposition of $n$ as $n =e^{i \langle D \rangle_{c_s} t  } N + e^{- i \langle D \rangle_{c_s} t  } \overline{N}   $.
The estimate of \eqref{model2}  is exactly the same as the estimate of \eqref{model} and we do not 
detail it again.    \vspace{.2cm} 

Now, it remains to control the last term on the right-hand side of \eqref{Nest}, namely the term 
involving the rest term. Again, if the low frequency term is non-outcome then we can 
estimate the term directly using  the integrable  $L^6$ decay  of the non-outcome part. 
Hence, the only  difficult terms are those for which the low frequency term is outcome. 
The most difficult  terms are very similar to those we treated above by integration by parts in time, namely 
\eqref{model}   and \eqref{model2}.  In addition we have terms of the type 
\begin{equation}\label{model3} 
\int_0^t  \int\! \int   e^{is\phi}  \widehat c (\eta)   \overline{ \widehat{\partial^N a} } (\xi -\eta) 
  {\widehat{\partial^N \overline{a}  }  } (\xi)  d\eta d\xi   ds 
\end{equation}  
where the phase $\phi$ is given by $\phi(\xi,\eta) = - \langle \xi - \eta \rangle_{c_s}  -  \langle \xi \rangle_{c_s}  
 \pm \langle \eta \rangle_{c_s}    $ and terms obtained by taking the complex conjugate  
These two types of terms are  even better than \eqref{model}  since integration by 
parts in time gains a factor $\frac1{|\phi|}$ that behaves like $\frac1{|\xi|}$ in the 
dangerous region $|\eta| << |\xi| $.

We also have terms  of the form 
\begin{equation}\label{model5} 
\int_0^t  \int\! \int   e^{is\phi}  \widehat c (\eta)  \widehat{\partial^{N-k} a} (\xi -\eta) 
 \overline{ {\widehat{\partial^N a}  }  } (\xi)  d\eta d\xi   ds 
\end{equation} 
where  $ k \geq 1 $  and 
the phase $\phi$ is given by $\phi(\xi,\eta) = \pm  \langle \xi - \eta \rangle_{c_s}  \pm   \langle \xi \rangle_{c_s}  
 \pm \langle \eta \rangle_{l}    $  and  $ e^{\pm it\<D\>_l} c(t) = \C(t)  $ denote outcome (low frequency) terms 
  and we denote undistinctly $\widehat{a}$ for $\widehat{a}$ or  $\widehat{\bar a}$ or their complex conjugate.
 From Lemma \ref{lemNon}, we know that $\phi$ is always bounded away from zero 
in the region we are interested in, nemaly $\xi$ large and $|\eta|\leq C_\mathcal{R} $. Integration by parts in time 
yields terms that are easier to control than above. In particular the corresponding term to 
\eqref{eng1}  can be  expressed as 
\begin{align}
    =& - \int_0^t  \int\! \int   \frac{1}{\phi}  \widehat \C(\eta)  
\left[ \widehat{u\cdot\nabla\partial^{N-k} \mathcal{A}    }  (\xi-\eta) \overline{ {\widehat{\partial^N \A}  }  } (\xi)   
 + \widehat{\partial^{N-k} \mathcal{A}    }  (\xi-\eta) \overline{ {\widehat{u\cdot\nabla\partial^N \A}  }  } (\xi)  
\right]  \label{cous1b} \\ 
&  - \int_0^t  \int\! \int   \frac{i}{\phi}  \widehat \C(\eta)  
\left[ \widehat{ n \langle D \rangle_{c_s}  \partial^{N-k} \mathcal{A}    }  
(\xi-\eta) \overline{ {\widehat{\partial^N \A}  }  } (\xi)   
  -  \widehat{\partial^{N-k} \mathcal{A}    }  (\xi-\eta) \overline{ {\widehat{  n \langle D \rangle_{c_s}  \partial^N \A}  }  } (\xi) 
\right]  \label{cous2b}    \\ 
 & - \int_0^t  \int\! \int   \frac{1}{\phi}  \widehat \C(\eta) 
  \left[ \widehat{R^{N-k}  }  (\xi-\eta) \overline{ {\widehat{\partial^N \A}  }  } (\xi)   
  -  \widehat{\partial^N \mathcal{A}    }  (\xi-\eta) \overline{ {\widehat{ R^N}   } }(\xi) 
 \right]  
 \end{align}
which can be easily estimated.

Finally, we also have terms for which $Pu$ carries the greatest number of derivatives. For these terms, we cannot use the cancellation coming 
from the energy  estimate. To gain the two factors of $|\xi|$, 
 we take advantage of the fact that 
$Pu$ is more regular, namely it is in $H^{N+1}$ and the fact that the phase $\phi$ involved 
in this case is bounded below by $|\xi|/C$.  The term corresponding to \eqref{model} is of  the form 

\begin{equation}\label{model6} 
\int_0^t  \int\! \int   e^{is\phi}  \widehat c (\eta)  \widehat{\partial^{N+1} b} (\xi -\eta) 
 \overline{ {\widehat{\partial^N a}  }  } (\xi)  d\eta d\xi   ds 
\end{equation} 
where the phase $\phi$ is given by $\phi(\xi,\eta) = \pm  \langle \xi - \eta \rangle_{1}  -   \langle \xi \rangle_{c_s}  
 \pm \langle \eta \rangle_{l}    $  and  $ e^{\pm it\<D\>_l} c(t) = \C(t)  $ denote outcome (low frequency) terms. It is clear that  $|\phi| \geq |\xi|/C$
in the region we are interested in, nemaly $\xi$ large and $|\eta|\leq C_\mathcal{R} $. Hence, we can perform 
an integration by parts in time and conclude as before. 

\section{Scattering}\label{scat}

Let us prove for instance that $\mathcal{A}$ scatters. We write symbolically the equation~(\ref{Aequat}) on $\mathcal{\A}$ as
$$
\partial_t \mathcal{A} - i \< D\>_{c_s} \mathcal{A} = \partial \mathcal{C} \mathcal{C}
$$
By definition, $\mathcal{A}$ will scatter in $H^{N-2}$, say at $+ \infty$, if and only if
$$
\int_0^t e^{i s  \< D\>_{c_s} }  \partial \mathcal{C}(s) \mathcal{C}(s) ds
$$
converges as $t \rightarrow \infty$. By the Strichartz estimates~(\ref{strichartz}), it suffices that the right-hand side $\partial \mathcal{C} \mathcal{C}$ belongs to $L_t^{\left( \frac{1}{2} + \frac{3}{2} \delta_1 \right)^{-1}} \left( [0,\infty), L_x^{\left( \frac{1}{3} - \delta_1 \right)^{-1}}\right) $. This is the case since
$$
\left\| \partial \mathcal{C} \mathcal{C} \right\|_{L_t^{\left( \frac{1}{2} + \frac{3}{2} \delta_1 \right)^{-1}} L_x^{\left( \frac{1}{3} - \delta_1 \right)^{-1}}} \lesssim \left\| \mathcal{C} \right\|_X^2 \left\| \< t \>^{C_0 \epsilon}  \< t \>^{-\frac{1}{2} - \frac{3}{2} \delta_1} \right\|_{L^{\left( \frac{1}{2} + \frac{3}{2} \delta_1 \right)^{-1}}} < \infty,
$$
where the last inequality follows since $\epsilon$ is small enough.

\section{Appendix: analytical tools}

\subsection{Sobolev embedding theorem}

If $1 \leq p \leq q \leq \infty$ and
$$
k > \frac{3}{p} - \frac{3}{q},
$$
then 
\begin{equation}
\label{sobolev}
\left\| f \right\|_q \lesssim \left\| f \right\|_{W^{k,q}}.
\end{equation}

\subsection{Product laws}

If $1 <p,r< \infty$, $1\leq q \leq \infty$, $k \geq 0$ and
$$
\frac{1}{p}+\frac{1}{q}=\frac{1}{r},
$$ 
then
\begin{equation}
\label{product}
\left\| fg \right\|_{W^{k,r}} \lesssim \left\|f \right\|_{W^{k,p}} \left\|g\right\|_{q} + \left\|f\right\|_{q} \left\|g \right\|_{W^{k,p}}
\end{equation}

\subsection{Dispersive and Strichartz estimates}

The standard dispersive estimates for Klein-Gordon can be found in Ginibre and Velo~\cite{GV}
\begin{equation}
\label{crane}
\left\| e^{it\<D\>} f \right\|_p \lesssim t^{\frac{3}{p}-\frac{3}{2}} 
\|f\|_{W^{ 5 \left( \frac{1}{2} - \frac{1}{p} \right) + \epsilon, p'}}
\;\;\;\;\;\;\mbox{if $2\leq p \leq \infty$ and $\epsilon>0$.}
\end{equation}
We need the following Strichartz estimate for the Klein-Gordon equation (see for instance Ibrahim, Masmoudi and Nakanishi~\cite{IMN}): if $\epsilon>0$ and $0 \leq \delta \leq \frac{1}{3}$,
\begin{equation}
\label{strichartz}
\left\| \int_0^t e^{is\<D\>} F(s)\,ds \right\|_2 \lesssim \left\|F\right\|_{L^{\left( \frac{1}{2}+\frac{3}{2}\delta \right)^{-1}} W^{\left( \frac{5}{6} - \frac{5}{2}\delta + \epsilon \right),\left( \frac{5}{6}-\delta \right)^{-1}}}.
\end{equation}
For the reader familiar with Besov spaces, this estimate follows from the interpolation between 
$$
\left\| \displaystyle \int_0^t e^{is\<D\>} F(s)\,ds \right\|_2 \lesssim \|F\|_{L^1 L^2}
$$
and
$$
\left\| \displaystyle \int_0^t e^{is\<D\>} F(s)\,ds \right\|_2 \lesssim \left\| F \right\|_{L^2 B^{5/6}_{6/5,2}}.
$$

\subsection{Boundedness of multilinear Fourier multipliers}

After cutting off with the help of the functions defined in Section~\ref{sectionmachinery}, the manipulations which we perform lead to various pseudo product operators. Their boundedness properties are stated in the following proposition. Notice that the statement below is very far from optimal, but sufficient for our purposes.

\begin{prop} 
\label{toucan} Assume that $m$ satisfies the estimates~(\ref{eider}).

(i)  Then for any $p,q,r$ in $(1,\infty)$ satisfying $\frac{1}{r} = \frac{1}{p} + \frac{1}{q}$, and $k \geq 0$,
$$
\left\| T_m (f,g) \right\|_{W^{k,r}} \lesssim \left\| f \right\|_{W^{k+1,p}} \left\| g \right\|_{W^{k+1,q}}.
$$

(ii) Assume
$$
\mu(\xi,\eta) = \widetilde{\chi}_\mathcal{O}(\xi) \chi_\mathcal{S}(\xi,\eta) m(\xi,\eta) \frac{1}{\phi} \quad\mbox{or} \quad \widetilde{\chi}_\mathcal{O}(\xi) \chi_\mathcal{T}(\xi,\eta) m(\xi,\eta) \frac{\partial_\eta \phi}{|\partial_\eta \phi|^2}
$$
Then there exists a constant, which we denote $n \geq 0$, such that for any $p,q,r$ in $(1,\infty)$ satisfying $\frac{1}{r} = \frac{1}{p} + \frac{1}{q}$, and $k \geq 0$,
$$
\left\| T_\mu (f,g) \right\|_{W^{k,r}} \lesssim \left\| f \right\|_{W^{k+n,p}} \left\|g\right\|_{W^{k+n,q}}.
$$

(iii) Assume
$$
\mu(\xi,\eta) = m(\xi,\eta) \zeta^1(\xi,\eta) \quad \mbox{or} \quad m(\xi,\eta) \zeta^1(\xi,\eta) \partial_\xi \phi(\xi,\eta).
$$
Then there exists a constant, which we stil denote $n \geq 0$, such that for any $p,q,r$ in $(1,\infty)$ satisfying $\frac{1}{r} = \frac{1}{p} + \frac{1}{q}$, and $k \geq 0$,
$$
\left\| T_\mu (f,g) \right\|_{W^{k,r}} \lesssim \left\| f \right\|_{W^{k+n,q}} \left\|g\right\|_{p}.
$$

(iv)  Assume
$$
\mu(\xi,\eta) = m(\xi,\eta) \zeta^2(\xi,\eta) \quad \mbox{or} \quad m(\xi,\eta) \zeta^2(\xi,\eta) \partial_\xi \phi(\xi,\eta).
$$
Then there exists a constant, which we stil denote $n \geq 0$, such that for any $p,q,r$ in $(1,\infty)$ satisfying $\frac{1}{r} = \frac{1}{p} + \frac{1}{q}$, and $k \geq 0$,
$$
\left\| T_\mu (f,g) \right\|_{W^{k,r}} \lesssim \left\| f \right\|_{p} \left\|g\right\|_{W^{k+n,p}}.
$$

(v) Assume
$$
\mu(\xi,\eta) \overset{def}{=} \frac{\chi_\mathcal{O}(\eta) \left(1-\theta \left( \frac{\xi-\eta}{M_0} \right) \right)\widetilde{m}(\eta,\xi)}{\phi(\xi,\eta)}
$$
where $\widetilde{m}$ and $\frac{1}{\phi}$ satisfy the estimates 
$$
|\partial_\xi^\alpha \partial_\eta^\beta \widetilde{m}(\xi,\eta)| \lesssim (|\xi|+|\eta|)^{-|\alpha|-|\beta|} \quad \mbox{and}\quad \left|\partial_\xi^\alpha \partial_\eta^\beta \frac{1}{\phi} \right| \lesssim (|\xi|+|\eta|)^{-|\alpha|-|\beta|-1}.
$$
Then
$$
\left\| T_\mu (f,g) \right\|_{H^k} \lesssim \|f\|_\infty \|g\|_{H^{k+1}}.
$$
\end{prop}

\begin{proof} Estimates similar to the first four points above were proved in~\cite{PG}. It essentially suffices to use the basic estimate
$$
\left\| T_\mu(f,g) \right\|_{r} \lesssim \|\mu\|_{H^{3/2+\epsilon}} \|f\|_p \|g\|_q
$$
if $\epsilon>0$ and $\frac{1}{p}+\frac{1}{q} = \frac{1}{r}$, the estimates given in Section~\ref{sectionmachinery} on the various symbols, and a paraproduct decomposition to handle large frequencies.

The fifth point follows from the classical Coifman-Meyer theorem~\cite{CM}.
\end{proof}

Next, we want to study a particular kind of symbol, which will not satisfy standard Coifman-Meyer bounds, but still admit H\"older-like bounds (in the bilinear case for instance, we only focus on the case $L^\infty \times L^2 \rightarrow L^2$ bound, but it should be clear from the proof that more general $L^p \times L^q \rightarrow L^r$ bounds, with $\frac{1}{p} + \frac{1}{q} = \frac{1}{r}$, also hold).

\begin{lem} \label{hummingbird}
Let $R$ be a fixed constant

(i) Let $\mu(\xi,\eta)$ be a smooth symbol such that 
$$
 \operatorname{Supp} \mu \subset \{ |\eta| \leq R \} \quad \mbox{and} \quad \left| \partial^\alpha_\xi \partial^\beta_\eta \mu(\xi,\eta) \right| \lesssim \frac{1}{|\xi|^{|\alpha|}} \quad \mbox{for any $\xi,\eta$}.
$$
Then $\displaystyle\left\|  T_\mu(f,g) \right\|_2 \lesssim \|f\|_\infty \|g\|_2$.

(ii) Let $\mu(\xi,\eta,\nu)$ be a smooth symbol such that 
$$
\operatorname{Supp} \mu \subset \{ |\eta| \leq R  \,,\, |\nu| \leq \frac{1}{200} |\xi| \} \quad \mbox{and} \quad \displaystyle \left| \partial^\alpha_\xi \partial^\beta_\eta \partial_\nu^\gamma \mu(\xi,\eta,\nu) \right| \lesssim \frac{1}{|\xi|^{|\alpha|+|\gamma|}}\quad \mbox{for any $\xi,\eta,\nu$}.
$$
Then $\displaystyle\left\|  T_\mu(f,g,h) \right\|_2 \lesssim \|f\|_\infty \|g\|_\infty \|h\|_2$.
\end{lem}

\begin{proof} We take for simplicity $R = 1$, and first define standard Fourier space decompositions
\begin{itemize}
\item Let $\zeta$ be a non-negative function, equal to 1 on $B(0,.9)$, zero outside of $B(0,2)$, and such that $\sum_{j \in \mathbb{Z}^3} \zeta(\xi - j) = 1$ for any $\xi$. Denote 
$$
Q_j \overset{def}{=} \sum_{j-3}^{j+3} \zeta(D-j).
$$
\item Let $\psi$ be a non-negative function, equal to 1 on $B(1,1.5)$, zero outside of $B(.5,4)$, and such that $\sum_{j \in \mathbb{Z}^3} \psi \left(  \frac{\xi}{2^j} \right) = 1$ for any $\xi \neq 0$. Further denote
$$
\widetilde{\psi}(\xi) = \sum_{j=-1}^{+1} \psi \left(   \frac{\xi}{2^j} \right) \quad \mbox{and} \quad \chi(\xi) = \sum_{j=-\infty}^{+1} \psi \left(   \frac{\xi}{2^j} \right)
$$
and the associated Fourier multipliers
$$
P_j = \psi \left( \frac{D}{2^j} \right) \quad,\quad
\widetilde{P}_j = \widetilde{\psi} \left( \frac{D}{2^j} \right) \quad \mbox{and} \quad S_j = \chi \left( \frac{D}{2^j} \right).
$$
\end{itemize}

\bigskip
\noindent
\underline{Proof of $(i)$} Split $\mu$ as follows
$$
\mu(\xi,\eta) = \sum_{j \in \mathbb{Z}^3} \zeta(\xi-j) \mu(\xi,\eta) \overset{def}{=} \sum_j \mu_j(\xi,\eta).
$$
The symbols $\mu_j$ are uniformly controlled in $\mathcal{C}^k$ for any $k$. Thus they define operators which are uniformly bounded $L^\infty \times L^2 \rightarrow L^2$. Observe furthermore that, due to frequency localization properties, $T_{\mu_j}(f,g) = T_{\mu_j}(f, Q_jg)$; and that for the same reason, the families $(T_{\mu_j}(f,g))_j$ and $(Q_j g)_j$ are almost orthogonal in $L^2$. These arguments lead to the following inequalities
$$
\left\| T_\mu(f,g) \right\|_2^2 \lesssim \sum_j \left\| T_{\mu_j} (f,g) \right\|_2^2 \lesssim \sum_j \left\| f \right\|_\infty^2 \left\| Q_j g \right\|_2^2 \lesssim \left\| f \right\|_\infty^2 \|g\|_2^2,
$$
proving $(i)$.

\bigskip
\noindent
\underline{Proof of $(ii)$} We will essentially run the original argument of Coifman and Meyer~\cite{CM}.
First set $\mu(\xi,\eta,\nu) = \widetilde{\mu}(\xi-\eta-\nu,\eta,\nu)$, and observe that the bounds on $\mu$ translate into
\begin{equation}
\label{newbound}
\left| \partial_\xi^\alpha \partial_\eta^\beta \partial_\nu^\gamma \widetilde{\mu}(\xi,\eta,\nu) \right|  \lesssim \frac{1}{|\xi|^{|\alpha|+|\gamma|}}.
\end{equation}
Next split $\widetilde{\mu}$ as follows
$$
\widetilde{\mu}(\xi,\eta,\nu) = \sum_j \psi \left( \frac{\xi}{2^j} \right) \widetilde{\mu} (\xi,\eta,\nu) \overset{def}{=} \sum_j \widetilde{\mu}_j(\xi,\eta,\nu) \quad \mbox{up to a remainder}.
$$
Since the remainder is compactly supported, and $\mu$ is closed, it will be easy to estimate, thus we forget about it and focus on the sum over $j$. The support of $\widetilde{\mu}_j(\xi,\eta,\nu)$ is contained in a box $\{ |\eta| \leq 1\,,\, |\xi| \leq 2^{j+1}\,,\, |\nu| \leq \frac{1}{200}2^{j+1} \}$. It can be expanded in (periodic) Fourier series adapted to the larger box $\{ |\eta| \leq 2\,,\,|\xi|\leq 2^{j+2}\,,\,|\nu| \leq 2^{j-7}$, and then recovered by restriction. This gives
$$
\widetilde{\mu}_j(\xi,\eta,\nu) = \widetilde{\psi} \left( \frac{\xi}{2^j} \right)  \chi(2\eta) \chi \left( \frac{\nu}{2^j-5} \right) \sum_{k,\ell,m \in \mathbb{Z}^3} \alpha_{k,\ell,m}^j  e^{i2\pi m \xi 2^{-j-2} }  e^{i 2\pi k\eta} e^{i 2\pi \ell \nu 2^{-j+7}},
$$
or, coming back to $\mu$,
$$
\mu_j(\xi,\eta,\nu) = \widetilde{\psi} \left( \frac{\xi-\eta-\nu}{2^j} \right)  \chi(2\eta) \chi \left( \frac{\nu}{2^j-5} \right) \sum_{k,\ell,m} \alpha_{k,\ell,m}^j  e^{i 2\pi m (\xi-\eta-\nu) 2^{-j-2} }  e^{i 2\pi k\eta} e^{i 2\pi \ell \nu 2^{-j+7}}.
$$
Next notice that the $\alpha^j_{k,\ell}$ are uniformly bounded in $j$, with arbitrarily quickly decaying (inverse) polynomial bounds:
\begin{equation}
\label{boundalpha}
\mbox{for any $N$,} \quad \sup_j |\alpha^j_{k,\ell,m}| = \alpha_{k,\ell,m} \lesssim \frac{1}{|(k,\ell,m)|^N}.
\end{equation}
This can be seen by simply coming back to their definition:
\begin{equation*}
\begin{split}
\alpha^j_{k,\ell,m} & = C 2^{-6j} \int_{|\xi|\leq 2^{j+2}} \int_{|\eta|\leq 1} \int_{|\nu|\leq 2^{j-7}} \widetilde{\mu}_j(\xi,\eta,\nu) e^{- i2\pi m \xi 2^{-j-2} }  e^{- i 2\pi k\eta} e^{- i 2 \pi \ell \nu 2^{-j+7}}\,d\nu\,\,d\eta\,d\xi \\
& = C \int_{|\xi|\leq 1} \int_{|\eta|\leq 1} \int_{|\nu|\leq 1} \widetilde{\mu}_j( 2^{j+2} \xi, \eta , 2^{j-7} \nu) e^{- i 2\pi m \xi}  e^{- i 2\pi k\eta} e^{- i2\pi \ell \nu }\,d\nu\,\,d\eta\,d\xi,
\end{split}
\end{equation*}
and the conclusion follows since the bounds~(\ref{newbound}) imply a uniform control (in $j$) of the symbols $\widetilde{\mu}_j( 2^{j+2} \xi, \eta , 2^{j-7} \nu)$.

Coming back to physical space, we have achieved the following decomposition for $T_\mu$:
$$
T_\mu(f,g,h) = \sum_{j,k,\ell,m} \alpha_{k,\ell,m}^j S_{j-5} f(\cdot+k) S_{0} g(\cdot + \ell 2^{-j+7}) \widetilde{P}_j h(\cdot + m 2^{-j-2}).
$$
The desired estimates follows easily by almost orthogonality between the $j$-summands
\begin{equation*}
\begin{split}
\left\| T_\mu(f,g,h) \right\|_2 & \lesssim \sum_{k,\ell,m} \left[ \sum_j  \alpha_{k,\ell,m}^j \left\| S_{j-5} f(\cdot+k) S_{0} g(\cdot + \ell 2^{-j+7}) \widetilde{P}_j h(\cdot + m 2^{-j-2}) \right\|_2^2 \right]^{1/2} \\
& \lesssim \sum_{k,\ell,m} \left[ \alpha_{k,\ell,m}^j \|f\|_\infty^2 \|g\|_\infty^2 \sum_j \| \widetilde{P}_j h(\cdot + m 2^{-j-2}) \|_2^2 \right]^{1/2} \\
& \lesssim \sum_{k,\ell,m} \alpha_{k,\ell,m} \left[  \|f\|_\infty^2 \|g\|_\infty^2 \|h\|_2^2 \right]^{1/2} \\
& \lesssim \|f\|_\infty \|g\|_\infty \|h\|_2,
\end{split}
\end{equation*}
where we used in the last inequality the bound~(\ref{boundalpha}). \end{proof}

Equipped with the previous lemma, we can prove the following proposition.

\begin{prop} 
\label{coot}
Let $M$ be as in~(\ref{egret}), and fix $\alpha>0$. Then the following estimate holds:
$$
\left\| T_M(f,g,h) \right\|_2 \lesssim \|f\|_{W^{1+\alpha,\infty}} \|g\|_\infty \|h\|_2.
$$
\end{prop}

\begin{proof}
Recall that 
$$
M(\xi,\eta,\nu) = - \mu(\xi,\eta)(\eta + \nu) + \left[ \mu(\xi,\eta) - \mu(\xi-\nu,\eta) \right] \xi \overset{def}{=} M_1(\xi,\eta,\nu) + M_2(\xi,\eta,\nu),
$$
where the operator $\mu$ is supported in a strip $\{ |\eta| \lesssim 1\,,\,|\xi|>>1\}$ and satisfies the bounds
$$
\left| \partial^\alpha_\xi \partial^\beta_\eta \mu  \right| \lesssim \frac{1}{|\xi|^{|\alpha|}}.
$$
We will treat separately the operators $T_{M_1}$ and $T_{M_2}$, further distinguishing for the latter between the regions where $|\xi-\eta-\nu| \lesssim \nu$, and those where $|\xi-\eta-\nu|>> |\nu|$, by writing
$$
T_{M_2}(f,g,h) = \sum_{j\geq 0} T_{M_2} (P_j f,g,S_{j+10} h) + \sum_{j \geq 0} T_{M_2} (S_{j-10} f,g,P_j h) \quad \mbox{up to a remainder}.
$$
Since the remainder is smooth and compactly supported, it is easily estimated, and we forget about it in the following in order to concentrate on the sum over $j$. Notice that we overtook the Littlewood-Paley operators $P_j$ and $S_j$ defined in Lemma~\ref{hummingbird}.

\bigskip
\noindent
\underline{The operator $T_{M_1}$.} Simply observe that
$$
T_{M_1}(f,g,h) = T_{\mu}(\nabla f, gh) + T_\mu(f,(\nabla g) h).
$$
Thus Lemma~\ref{hummingbird} gives the conclusion.

\bigskip
\noindent
\underline{The operator $T_{M_2}$ in the case $|\xi-\eta-\nu| \lesssim \nu$.} Recall that $M_2$ is given by $\left[ \mu(\xi,\eta) - \mu(\xi-\nu,\eta) \right] \xi$. There is no cancellation between the two summands in the range we consider, so $\mu(\xi,\eta) \xi$ and $\mu(\xi-\nu,\eta) \xi$ can be considered separately. Since they are estimated in similar ways, we focus on the first one. Notice that $T_{\mu(\xi,\eta) \xi} = \nabla T_\mu$. Using Lemma~\ref{hummingbird} and proceeding in a straightforward way, we get the desired estimates:
\begin{equation*}
\begin{split}
\left\| \sum_j \nabla T_\mu(P_j f, g S_{j+10} h) \right\|_2 & \lesssim \sum_j 2^j \left\| T_\mu(P_j f, g S_{j+10} h) \right\|_2 \lesssim \sum_j 2^j \|P_j f\|_\infty \|g\|_\infty \|S_j h\|_2\\
& \lesssim \|f\|_{W^{1+\alpha,\infty}} \|g\|_2 \|h\|_2.
\end{split}
\end{equation*}

\bigskip
\noindent
\underline{The operator $T_{M_2}$ in the case $|\xi-\eta-\nu| >> \nu$.} In this case, we observe that the operator 
$$
(f,g,h) \mapsto \sum_j T_{M_2} (S_{j-10} f,g,P_j h)
$$
has a symbol
$$
M_2'(\xi,\eta,\nu) = M_2(\xi,\eta,\nu) \sum_j \psi \left(\frac{\xi-\eta-\nu}{2^j} \right) \chi \left( \frac{\nu}{2^{j-10}} \right)
$$
which can be written
$$
M_2'(\xi,\eta,\nu) = \widetilde{M}_2(\xi,\eta,\nu) \cdot \nu \quad \mbox{with} \quad 
\widetilde{M}_2(\xi,\eta,\nu) \overset{def}{=} \sum_j \psi \left(\frac{\xi-\eta-\nu}{2^{j-10}} \right) \chi \left( \frac{\nu}{2^j} \right) \xi \int_0^1 \partial_\xi \mu(\xi-t\nu,\eta) \,dt.
$$
The key observation is that, due to the hypotheses on $\mu$,  we have the bound 
$|\xi |^{|\alpha|} |  \partial^\alpha_\xi \mu | \leq C$   and hence 
$\widetilde{M}_2$ satisfies the conditions of Lemma~\ref{hummingbird}. The estimate follows easily:
$$
\left\| T_{M_2} (S_j f,g,P_j h) \right\|_2 = \left\| T_{\widetilde{M}_2} (\nabla f, g,h) \right\|_2 \lesssim \left\|\nabla f\right\|_\infty \left\|g\right\|_\infty \left\|h\right\|_2.
$$
\end{proof}

\end{document}